\documentclass[10pt,psamsfonts]{article}

\usepackage{amsmath, amsfonts, amsthm, amssymb, amscd, enumerate, url, slashed, stmaryrd, upgreek, thmtools, mathdots}
\usepackage[all,arc,color]{xy}
\usepackage{enumerate}
\usepackage{mathrsfs}
\usepackage{mathtools}
\usepackage[utf8]{inputenc}
\usepackage{tikz-cd}

   \usepackage[colorlinks = true,
            linkcolor = blue,
            urlcolor  = cyan,
            citecolor = red,
            anchorcolor = blue,
            pagebackref]{hyperref}
            
\usepackage[alphabetic,backrefs]{amsrefs}
 
 \usepackage[parfill]{parskip}
 \usepackage{anysize}
 \marginsize{1in}{1in}{1in}{1in}
 
\begingroup
 \makeatletter
 \@for\theoremstyle:=definition,remark,plain\do{%
 \expandafter\g@addto@macro\csname th@\theoremstyle\endcsname{%
 \addtolength\thm@preskip\parskip
 }%
 }
 \endgroup

\declaretheorem[name=Theorem,numberwithin=section]{thm}
\declaretheorem[name=Proposition,numberlike=thm]{prop}
\declaretheorem[name=Lemma,numberlike=thm]{lemma}

\declaretheorem[name=Definition,style=definition,qed=$\blacktriangle$,numberlike=thm]{defn}

\declaretheorem[name=Remark,style=definition,qed=$\blacktriangle$,numberlike=thm]{rmk}

\newcounter{commentCounter}
\newcommand{\sta}{\star}

\newcommand{\cG}{\mathcal G}

\newcommand{\dive}{\operatorname{div}}

\newcommand{\pa}[2]{\frac{\partial #1}{\partial #2}}
\newcommand{\paop}[1]{\frac{\partial}{\partial #1}}
\newcommand{\cV}{\mathcal{V}}
\newcommand{\cE}{\mathcal{E}}

\newcommand{\Vol}{\mathsf{Vol}}
\newcommand{\vol}{\mathsf{vol}}

\newcommand{\real}{\mathrm{Re}}

\newcommand{\G}{\mathrm{G}_2}
\newcommand{\Spin}[1]{\mathrm{Spin}(#1)}

\newcommand{\U}[1]{\mathrm{U}(#1)}
\newcommand{\SU}[1]{\mathrm{SU}(#1)}

\newcommand{\ph}{\varphi}
\newcommand{\ps}{\psi}
\newcommand{\Ph}{\Phi}

\newcommand{\hk}{\mathbin{\hbox{\vrule height0.3pt width5pt depth 0.2pt \vrule height5pt width0.4pt depth 0.2pt}}}

\newcommand{\tr}{\operatorname{Tr}}

\newcommand{\ddx}[1]{\frac{\del}{\del x^{#1}}}
\newcommand{\dx}[1]{d x^{#1}}

\newcommand{\del}{\partial}
\newcommand{\delbar}{\overline{\partial}}
\newcommand{\ddt}{\frac{d}{d t}}

\newcommand{\eps}{\varepsilon}
\newcommand{\dist}{\operatorname{dist}}

\newcommand{\rest}[2]{ {\left. {#1} \right|}_{{#2} }}

\newcommand{\wt}{\widetilde}
\newcommand{\wh}{\widehat}

\newcommand{\ol}{\overline}
\newcommand{\lf}{\ell(\ol{\nabla} F)}

\newcommand{\w}{\wedge}

\allowdisplaybreaks

\begin{document}

\title{A variational characterization of calibrated submanifolds}

\author{Da Rong Cheng \\ \textit{Department of Pure Mathematics, University of Waterloo} \\ \tt{drcheng@uwaterloo.ca} \and Spiro Karigiannis \\ \textit{Department of Pure Mathematics, University of Waterloo} \\ \tt{karigiannis@uwaterloo.ca} \and Jesse Madnick \\ \textit{National Center for Theoretical Sciences, National Taiwan University} \\ \tt{jmadnick@ncts.ntu.edu.tw}}

\maketitle

\begin{abstract}
Let $M$ be a fixed compact oriented embedded submanifold of a manifold $\ol{M}$. Consider the volume $\cV (\ol{g}) = \int_M \vol_{(M, g)}$ as a functional of the ambient metric $\ol{g}$ on $\ol{M}$, where $g = \rest{\ol{g}}{M}$. We show that $\ol{g}$ is a critical point of $\cV$ with respect to a special class of variations of $\ol{g}$, obtained by varying a calibration $\mu$ on $\ol{M}$ in a particular way, if and only if $M$ is calibrated by $\mu$. We do not assume that the calibration is closed. We prove this for almost complex, associative, coassociative, and Cayley calibrations, generalizing earlier work of Arezzo--Sun in the almost K\"ahler case. The Cayley case turns out to be particularly interesting, as it behaves quite differently from the others. We also apply these results to obtain a variational characterization of Smith maps.
\end{abstract}

\tableofcontents

\section{Introduction} \label{sec:intro}

Let $(\ol{M}, \ol{g})$ be an $n$-dimensional Riemannian manifold, and let $M$ be a $k$-dimensional compact oriented manifold. Given an immersion $\iota \colon M \to \ol{M}$, there is an induced metric $g = \iota^* \ol{g}$ on $M$, and from $g$ and the orientation on $M$ there is a volume form $\vol_{(M, g)}$. Thus we can compute the volume $\int_M \vol_{(M, g)}$ of $M$ with respect to $g$. This is a functional on the space of immersions $\iota \colon M \to \ol{M}$, and it is well-known that the critical points are the \emph{minimal} immersions --- that is, immersions with vanishing mean curvature.

In this paper we consider the following related but different situation. We \emph{fix} the map $\iota \colon M \to \ol{M}$, which for technical reasons we take to be an embedding, so that $M \subset \ol{M}$ and $\iota$ is the inclusion. Instead, what we vary is the \emph{ambient metric} $\ol{g}$ on $\ol{M}$. Thus, the map $\ol{g} \mapsto \int_M \vol_{(M, g)}$ where $g = \rest{\ol{g}}{M}$ is a functional on the space of metrics $\ol{g}$ on $\ol{M}$. For general variations of $\ol{g}$, this functional has no critical points.

Suppose $(\ol{M}, \ol{g})$ is equipped with an additional geometric structure, namely a \emph{calibration} $k$-form $\mu$. (Here by calibration we just mean a $k$-form with comass one. It need not be parallel or even closed.) The submanifold $M$ is \emph{calibrated} by $\mu$ if $\rest{\mu}{M} = \vol_{(M, g)}$. If $d \mu = 0$, then this implies that $M$ is homologically volume-minimizing, and hence minimal. In general if $d \mu \neq 0$, the calibrated submanifolds are just a distinguished class of $k$-dimensional oriented submanifolds determined by the $k$-form $\mu$.

The examples we consider are the calibrations associated to $\U{m}$, $\G$, or $\Spin{7}$-structures. In each of these cases we can consider a particular \emph{subclass} $\cG$ of variations of the ambient metric $\ol{g}$ on $\ol{M}$ that are induced by varying the calibration forms in a particular way. The precise nature of the metric variations in this subclass $\cG$ is slightly different in each of the cases we consider, and is explained in detail in each of the relevant subsections of \S\ref{sec:calibrated}. We prove the following:

\textbf{Meta-Theorem.} The submanifold $M$ is \emph{calibrated} by $\mu$ on $(\ol{M}, \ol{g})$ if and only if $\ol{g}$ is \emph{critical} for the functional $\ol{g} \mapsto \int_M \vol_{(M, g)}$ with respect to variations of $\ol{g}$ in the subclass $\cG$ determined by $\mu$.

This gives a \emph{variational characterization} for these types of calibrated submanifolds. Precise statements are given in Theorems~\ref{thm:Um-A},~\ref{thm:Um-B},~\ref{thm:assoc-A},~\ref{thm:assoc-B},~\ref{thm:coassoc-A},~\ref{thm:coassoc-B},~\ref{thm:Cayley-A}, and~\ref{thm:Cayley-B}.

In the case when $\ol{M}$ is \emph{almost K\"ahler} (that is, when $\ol{M}$ is equipped with a $\U{m}$-structure with \emph{closed} K\"ahler form $\omega$), this result was originally proved when $k=2$ by Arezzo--Sun in~\cite[Theorem 1.3, Proposition 2.3]{AZ1}. For $k > 2$, one direction (that critical implies calibrated) was proved by Arezzo--Sun in~\cite[Theorem 1.2]{AZ2}. This latter result was in the more general setting of currents and geometric measure theory, but still assumed the closedness of the K\"ahler form $\omega$.

\begin{rmk} \label{rmk:similar-work}
Similar results were obtained by Sun~\cite{S} and Tan~\cite{T} for a different functional, namely the integral of the $p^{\text{th}}$ power of the calibration angle. Sun first considered a target that was a K\"ahler manifold, which Tan later generalized to an almost Hermitian manifold.
\end{rmk}

We generalize the work of Arezzo--Sun in two ways. First, in the $\U{m}$ case we remove the restriction that $d \omega = 0$, and also prove the converse when $k > 2$. Second, we prove similar results for the associative, coassociative, and Cayley calibrations for $\G$ and $\Spin{7}$-structures.

There is one technical caveat that is somewhat mysterious. The special class $\cG$ of ambient metric variations in the $\Spin{7}$ case \emph{depends} on the submanifold $M$. A similar phenomenon occurs when $k > 2$ in the $\U{m}$ case, but by imposing $d \omega = 0$ we show that we get a class $\cG$ of metric variations that is independent of $M$. By contrast, we show that imposing closedness of the calibration form in the $\Spin{7}$ case \emph{cannot} remove the dependence on $M$ of the class $\cG$. (See Remarks~\ref{rmk:Um-A-b} and~\ref{rmk:Cayley-cannot-fix}.)

In Section~\ref{sec:calibrated}, we first discuss some preliminaries common to all four cases, and then treat each of the four cases in detail. In Section~\ref{sec:further} we make some further observations, including an alternative argument for the ``only if'' part of the Meta-Theorem, the relation to a similar variational characterization of minimal submanifolds, an application to Smith maps, and future questions.

\textbf{Acknowledgements.} The authors are grateful to the anonymous referee for alerting our attention to the references~\cite{S} and~\cite{T}.

\section{A variational characterization of calibrated submanifolds} \label{sec:calibrated}

In this section we establish a variational characterization of calibrated submanifolds for $\U{m}$, $\G$, and $\Spin{7}$-structures. We begin with some preliminaries, and then treat each case.

\subsection{Preliminaries} \label{sec:prelim}

Let $(\ol{M}, \ol{g})$ be an $n$-dimensional Riemannian manifold, and let $M$ be a compact oriented embedded $k$-dimensional submanifold of $M$. The metric $\ol{g}$ on $\ol{M}$ induces a Riemannian metric $g = \rest{\ol{g}}{M}$ on $M$ by restriction, and hence a Riemannian volume form $\vol_{(M,g)}$ from $g$ and the orientation on $M$. The volume of $M$ is thus $\int_M \vol_{(M,g)}$.

Whenever we use a local orthonormal frame $\{ e_1, \ldots, e_n \}$ of $(\ol{M}, \ol{g})$, it is always chosen to be adapted to the submanifold $M$, meaning that at any points of $M$, the first $k$ elements form a local orthonormal frame $\{ e_1, \ldots, e_k \}$ of $TM$. (Except in the coassociative case it is more convenient to list the normal vectors before the tangent vectors.) We use the Einstein summation convention throughout. For a general vector field $X$ along $M$ we use $X^T$ and $X^{\perp}$ to denote the tangent and normal parts, respectively. We also use $\ol{g}$ to identify vector fields and $1$-forms on $\ol{M}$ in the usual way, and we sometimes use $\langle \cdot, \cdot \rangle$ to denote the metric $\ol{g}$. Finally, we write $\ol{\nabla}$ for the Levi-Civita connection on $\ol{M}$ with respect to $\ol{g}$, and $\nabla$ for the Levi-Civita connection on $M$ with respect to $g = \rest{\ol{g}}{M}$.

We have a functional $\cV$ on the space of Riemannian metrics on $\ol{M}$ given by
\begin{equation} \label{eq:functional}
\ol{g} \mapsto \cV(\ol{g}) = \int_M \vol_{(M,g)} \qquad \text{ where $g = \rest{\ol{g}}{M}$}.
\end{equation}
The following is a standard computation that gives the first variation formula for $\cV$.
\begin{lemma} \label{lemma:first-variation}
Let $\ol{g}_t$ be a 1-parameter family of metrics on $\ol{M}$ with $\ddt \ol{g}_t = h_t$. Let $g_t = \rest{\ol{g}_t}{M}$. Then we have
\begin{equation} \label{eq:first-variation}
\ddt \cV (\ol{g}_t) = \frac{1}{2} \int_M (\tr_{g_t} h_t) \, \vol_{(M, g_t)}.
\end{equation}
\end{lemma}
\begin{proof}
Let $\{ e_1, \ldots, e_k \}$ be any local frame for $TM$, independent of $t$. This determines the local matrix of functions
\begin{equation} \label{eq:first-v-temp}
(g_t)_{ab} = g_t (e_a, e_b) = \ol{g}_t (e_a, e_b),
\end{equation}
with respect to which we have $\vol_{(M, g_t)} = \sqrt{ \det (g_t) } \, e^1 \w \cdots \w e^k$. From the usual formula for the derivative of the determinant, we get
\begin{align*}
\ddt \sqrt{ \det g_t } & = \frac{1}{2 \sqrt{ \det g_t }} \ddt \big( \det g_t \big) = \frac{1}{2 \sqrt{ \det g_t }} (g_t)^{ab} \Big( \ddt (g_t)_{ab} \Big) \big( \det g_t \big) \\
& = \frac{1}{2} \sqrt{ \det g_t } \, (g_t)^{ab} \Big( \ddt (g_t)_{ab} \Big).
\end{align*}
From~\eqref{eq:first-v-temp} we get $\ddt (g_t)_{ab} = \ddt \ol{g}_t (e_a, e_b) = h_t (e_a, e_b) = (h_t)_{ab}$. Thus we find that
$$ \ddt \vol_{(M, g_t)} = \frac{1}{2} (g_t)^{ab} (h_t)_{ab} \, \vol_{(M, g_t)} = \frac{1}{2} (\tr_{g_t} h_t) \, \vol_{(M, g_t)}.$$
Since $\ddt \cV (\ol{g}_t) = \int_M \ddt \vol_{(M, g_t)}$, equation~\eqref{eq:first-variation} follows.
\end{proof}

It is clear that if we allow all possible variations $\ol{g}_t$ of the ambient metric $\ol{g}$ on $\ol{M}$, then the functional~\eqref{eq:functional} has no critical points. For example, we can consider scalings $\ol{g}_t = e^t \ol{g}$, which gives $h_t = \ol{g}_t$, and then~\eqref{eq:first-variation} yields $\ddt \cV (\ol{g}_t) = \frac{1}{2} k \, \Vol(M, g_t) > 0$. So in order to use the functional~\eqref{eq:functional} to give a variational characterization of certain submanifolds of $(\ol{M}, \ol{g})$, we must necessarily consider restricting the allowed variations of the ambient metric $\ol{g}$.

When the ambient manifold $\ol{M}$ is equipped with a certain $G$-structure, such as a $\U{m}$, $\G$, or $\Spin{7}$-structure, then the ambient metric $\ol{g}$ is induced by a special type of differential form $\mu$ on $\ol{M}$. (Actually in the $\U{m}$ case there is also an independent almost complex structure.) Thus we can consider a subclass $\cG$ of variations of the ambient metric $\ol{g}$ on $\ol{M}$ induced by varying the form $\mu$ in a particular way. We consider four specific cases, corresponding to the usual comass one differential forms for $\U{m}$, $\G$, and $\Spin{7}$-structures. That is, these are differential forms which are \emph{calibrations} in the sense of Harvey--Lawson when they are closed.

\begin{defn} \label{defn:structures}
The four cases we consider are the following (details can be found in~\cites{HL, LL}).
\begin{enumerate}[{$[$}1{$]$}]
\item $\U{m}$-structure (almost complex submanifolds): $(\ol{M}, \ol{g})$ is $2m$-dimensional, equipped with an almost complex structure $J$ that is orthogonal with respect to $\ol{g}$, and $\omega(X, Y) = \ol{g} (J X, Y)$ is the associated $(1,1)$-form. We do not assume that $J$ is integrable, nor that $\omega$ is closed. The classical Wirtinger inequality says that $\frac{1}{k!} \omega^k$ has comass one for all $1 \leq k \leq m-1$. A $2k$-dimensional orientable submanifold $M$ of $\ol{M}$ is \emph{almost complex}, also called \emph{$J$-holomorphic}, if its tangent spaces $T_x M$ are invariant under $J$. This is equivalent to the existence of an orientation on $M$ for which $M$ is calibrated by $\frac{1}{k!} \omega^k$.
\item $\G$-structure: $(\ol{M}, \ol{g})$ is $7$-dimensional, equipped with a $\G$-structure $\ph$ (a positive $3$-form) compatible with $\ol{g}$. The pair $(\ph, \ol{g})$ determines an orientation and Hodge dual $4$-form $\ps = \sta_{\ph} \ph$. We make no assumption on the torsion of the $\G$-structure. It is well-known that both $\ph$ and $\ps$ have comass one.
\begin{enumerate}[{$[$}a{$]$}]
\item (associative submanifolds): We have a vector cross product $\times$ given by $\ph(X, Y, Z) = \ol{g} (X \times Y, Z)$. A $3$-dimensional orientable submanifold $M$ of $\ol{M}$ is \emph{associative}, if its tangent spaces $T_x M$ are invariant under $\times$. This is equivalent to the existence of an orientation on $M$ for which $M$ is calibrated by $\ph$.
\item (coassociative submanifolds): A $4$-dimensional orientable submanifold $M$ of $\ol{M}$ is \emph{coassociative} if $\rest{\ph}{M} = 0$. This is equivalent to the existence of an orientation on $M$ for which $M$ is calibrated by $\ps$.
\end{enumerate}
\item $\Spin{7}$-structure (Cayley submanifolds): $(\ol{M}, \ol{g})$ is $8$-dimensional, equipped with a $\Spin{7}$-structure $\Ph$ (a positive $4$-form) compatible with $\ol{g}$. The pair $(\Ph, \ol{g})$ determines the vector cross product $P$ given by $\Ph(X, Y, Z, W) = \ol{g} (P(X, Y, Z), W)$. We make no assumption on the torsion of the $\Spin{7}$-structure. It is well-known that $\Ph$ has comass one. A $4$-dimensional orientable submanifold $M$ of $\ol{M}$ is \emph{Cayley} if its tangent spaces $T_x M$ are invariant under $P$. This is equivalent to the existence of an orientation on $M$ for which $M$ is calibrated by $\Ph$. \qedhere
\end{enumerate}
\end{defn}

In the list above, the case $k=1$ of [1] and cases [2a] and [3] correspond to those calibrations associated to nontrivial vector cross products~\cites{CKM, Gr-CR}. Their corresponding calibrated submanifolds are called \emph{instantons} by Lee--Leung~\cite{LL}. Case [2b] of coassociative submanifolds are \emph{branes} in the language of Lee--Leung~\cite{LL}. Because the coassociative calibration does not arise from a vector cross product, we need an additional result to establish our variational characterization in this case, namely, Proposition~\ref{prop:coassoc-condition}. Another class of calibrated submanifolds which are branes is the class of \emph{special Lagrangian submanifolds} in manifolds with $\SU{m}$-structure. See Section~\ref{sec:future} for a brief discussion on the special Lagrangian case, and why we did not consider it in this paper.

In Sections~\ref{sec:Um},~\ref{sec:G2}, and~\ref{sec:Spin7} we prove two theorems, in each of these four specific cases. \emph{Informally}, the two theorems are of the following general form in each case, albeit with some technical caveats.

\textbf{Theorem A}. If $M$ is \emph{calibrated} by $\mu$, then $\ol{g}$ is a critical point of $\cV$ with respect to variations in the subclass $\cG$.

\textbf{Theorem B.} If $\ol{g}$ is a critical point of $\cV$ with respect to variations in the subclass $\cG$, then $M$ is \emph{calibrated} by $\mu$.

These results appear in Theorems~\ref{thm:Um-A} and~\ref{thm:Um-B} for the almost complex case, Theorems~\ref{thm:assoc-A} and~\ref{thm:assoc-B} for the associative case, Theorems~\ref{thm:coassoc-A} and~\ref{thm:coassoc-B} for the coassociative case, and Theorems~\ref{thm:Cayley-A}  and~\ref{thm:Cayley-B} for the Cayley case.

\begin{rmk} \label{rmk:B-stronger}
Although Theorem B as stated is precisely a converse to Theorem A, our proof actually demonstrates more. Namely, we show that there is a \emph{particular} family of metric variations in the subclass $\cG$ such that if $\cV$ has vanishing first variation in these particular directions, then $M$ is calibrated by $\mu$. It then follows from Theorem A that $\cV$ must in fact have vanishing first variation with respect to \emph{all} directions in $\cG$.
\end{rmk}

\begin{rmk} \label{rmk:up-to-orientation}
In fact, the conclusion of Theorem B is that there exists an orientation on $M$ with respect to which $M$ is calibrated by $\mu$. Neither Theorem A nor Theorem B depend on the given orientation on $M$ in any significant way.
\end{rmk}

The following result from Riemannian geometry is used in all four cases to prove Theorem B.

\begin{defn} \label{defn:F}
For $x \in \ol{M}$, let $\delta(x) = \dist_{\ol{g}}(M, x)$ denote the distance with respect to $\ol{g}$ from $x$ to the embedded submanifold $M$. Since $M$ is compact, there is an open neighbourhood $U$ of $M$ in $\ol{M}$ on which the function $\frac{1}{2} \delta^2$ is smooth. We denote by $F \in C^{\infty} (\ol{M})$ a smooth nonnegative extension of $\frac{1}{2} \delta^2$ to the whole of $\ol{M}$. Let $H$ denote the Hessian $\ol{\nabla}^2 F$ of $F$ with respect to $\ol{g}$, which is a symmetric 2-tensor on $\ol{M}$, so $H_{ij} = \ol{\nabla}_i \ol{\nabla}_j F$.
\end{defn}

\begin{lemma} \label{lemma:distance-function}
Let $F \in C^{\infty} (\ol{M})$ and $H = \ol{\nabla}^2 F$ be as in Definition~\ref{defn:F}. Then we have
\begin{equation} \label{eq:gradient-F}
\ol{\nabla} F = 0 \quad \text{along the submanifold $M$}.
\end{equation}
Moreover, we have
\begin{equation} \label{eq:Hessian-F}
(\ol{\nabla}^2 F) (X, Y) = H (X, Y) = \ol{g} (X^{\perp}, Y^{\perp}) \quad \text{along the submanifold $M$}.
\end{equation}
Via the musical isomorphism of $\ol{g}$, we can also write~\eqref{eq:Hessian-F} in the useful form
\begin{equation} \label{eq:Hessian-Fb}
\ol{\nabla}_X (\ol{\nabla} F) = X^{\perp} \quad \quad \text{along the submanifold $M$}.
\end{equation}
\end{lemma}
\begin{proof}
Since $F \geq 0$ everywhere on $\ol{M}$ and $F \equiv 0$ on $M$, we see that $F$ has an absolute minimum at all points on $M$ and thus $\ol{\nabla} F$ must vanish on $M$, establishing~\eqref{eq:gradient-F}.

The identity~\eqref{eq:Hessian-F} was stated in~\cite[Proposition 2.5]{AZ1} without proof. We outline a proof here for completeness. Inside a tubular neighbourhood $U$ of $M$ in $\ol{M}$, we can consider \emph{local Fermi coordinates}, which are defined as follows. Fix local coordinates $x^1, \ldots, x^k$ for $M$ centred at $x \in M$, and let $\{ \nu_1, \ldots, \nu_{n-k} \}$ be a local orthonormal frame for the normal bundle $NM$ defined on the same neighbourhood of $x$ in $M$. The Riemannian exponential map $\exp$ of $(\ol{M}, \ol{g})$ determines the local Fermi coordinates $(x, y) = (x^1, \ldots, x^k, y^1, \ldots y^{n-k})$, where the inverse of the coordinate chart is given by
$$ \Upsilon (x, y) = \exp_{x} \Big( \sum_{\alpha=1}^{n-k} y^{\alpha} \nu_{\alpha} \Big). $$
It is easy to show that~\cite[Proposition 5.26]{Lee-Riem} we have
\begin{equation} \label{eq:Hess-id-1}
\Upsilon^* \ol{g} = \begin{pmatrix} \ast & 0 \\ 0 & \delta_{\alpha \beta} \end{pmatrix} \quad \text{ at $x \in M$}
\end{equation}
where $\ast$ is some irrelevant $k \times k$ block matrix, and $\delta_{\alpha \beta}$ is the $(n-k) \times (n-k)$ identity matrix. Moreover, the Christoffel symbols $\Gamma^l_{ij}$ of $\ol{\nabla}$, evaluated at $x \in M$, vanish for $k+1 \leq i, j \leq n$. By the properties of the exponential map, in these coordinates the function $F = \frac{1}{2} \delta^2$ is $F(x,y) = \frac{1}{2} \sum_{\alpha=1}^{n-k} (y^{\alpha})^2$. The Hessian $H = \ol{\nabla}^2 F$ is $H_{ij} = \partial_i \partial_j F - \Gamma^l_{ij} \partial_l F$, and thus we have
\begin{equation} \label{eq:Hess-id-2}
\Upsilon^* H = \begin{pmatrix} 0 & 0 \\ 0 & \delta_{\alpha \beta} \end{pmatrix} \quad \text{ at $x \in M$}.
\end{equation}
Comparing the expressions~\eqref{eq:Hess-id-1} and~\eqref{eq:Hess-id-2} yields the identity~\eqref{eq:Hessian-F}.
\end{proof}

\begin{rmk} \label{rmk:lf}
We frequently have occasion to compute expressions that involve a complicated term that is linear in $\ol{\nabla} F$. Because such terms vanish on the submanifold $M$ by~\eqref{eq:gradient-F}, we do not have to keep careful track of them. Therefore, we henceforth write $\lf$ for terms that are linear in $\ol{\nabla} F$.
\end{rmk}

\subsection{Almost complex submanifolds of manifolds with $\U{m}$-structure} \label{sec:Um}

In this section we consider the case when $\ol{M}$ has a $\U{m}$-structure. Suppose that $n=2m$ is even and let $(\ol{g}, J, \omega)$ be a $\U{m}$-structure on $\ol{M}$. Thus $J$ is an almost complex structure that is orthogonal with respect to $\ol{g}$ (we do \emph{not} assume that $J$ is integrable), and $ \omega(X, Y) = \ol{g}(JX, Y)$ is the associated $(1,1)$-form, which \emph{need not} be closed. It follows that
\begin{equation} \label{eq:Um-fundamental}
\ol{g}(X, Y) = \omega(X, JY).
\end{equation}
Let $M$ be a compact oriented embedded $2k$-dimensional submanifold of $\ol{M}$ where $1 \leq k \leq m-1$. We consider the functional $\cV$ of~\eqref{eq:functional} in this setting, and consider the following restricted class of variations of the ambient metric $\ol{g}$.

Let $\alpha_t$ be a family of compactly supported smooth $1$-forms on $\ol{M}$ with $\alpha_{t=0} = 0$. Then
\begin{equation} \label{eq:Um-omegat}
\wt \omega_t = \omega + d \alpha_t
\end{equation}
is a one-parameter family of smooth $2$-forms on $\ol{M}$ with $\wt \omega_{t=0} = \omega$. Note that if $d \omega = 0$, then $\wt \omega_t$ lies in the same cohomology class as $\omega$ for all $t$, but we do not assume that $d \omega = 0$.

Note that the almost complex structure $J$ is not varying. Consider the type $(p,q)$ decomposition of $\wt \omega_t$. Let
\begin{equation} \label{eq:Um-omegat-2}
\omega_t = \wt \omega_t^{(1,1)}, \qquad \rho_t = \wt \omega_t^{(2,0) + (0,2)}.
\end{equation}
Then we have
$$ \omega_t (JX, JY) = \omega_t (X, Y), \qquad \rho_t (JX, JY) = - \rho_t (X, Y), $$
from which it follows that
$$ \text{$\omega_t (X, JY)$ is symmetric in $X, Y$ and $\rho_t (X, JY)$ is skew-symmetric in $X, Y$.} $$

Given the family $\wt \omega_t = \omega_t + \rho_t$, we define a new symmetric $2$-tensor $\ol{g}_t$ on $\ol{M}$ by
\begin{equation} \label{eq:Um-metric}
\begin{aligned}
\ol{g}_t (X, Y) = \omega_t (X, JY) & = \frac{1}{2} \big( \omega_t (X, JY) + \omega_t (Y, JX) \big) \\
& = \frac{1}{2} \big( \wt \omega_t (X, JY) + \wt \omega_t (Y, JX) \big).
\end{aligned}
\end{equation}
It is clear $\ol{g}_t (JX, JY) = \ol{g}_t (X, Y)$, so $J$ is $\ol{g}_t$ orthogonal for all $t$. When $t=0$ we get from~\eqref{eq:Um-fundamental} that $\ol{g}_{t=0} = \ol{g}$. Thus, for $|t|$ sufficiently small, $\ol{g}_t$ is positive definite. Moreover, we have
$$ \omega_t (X, Y) = \ol{g}_t (JX, Y), $$
so the triple $(\ol{g}_t, J, \omega_t)$ is a $\U{m}$-structure on $\ol{M}$ for all $|t|$ sufficiently small. We remark that the $(2,0) + (0,2)$ component $\rho_t$ of $\wt \omega_t$ does not contribute to the variation of the metric $\ol{g}_t$. A similar situation occurs in the associative, coassociative, and Cayley cases later in the paper.

Write $\dot{\alpha}_t$ for $\ddt \alpha_t$. We therefore have
$$ \ddt \wt \omega_t =  d \dot{\alpha}_t. $$
From~\eqref{eq:Um-metric} we have $\ol{g}_t (X, X) = \omega_t (X, JX) = \wt \omega_t (X, JX)$, and thus from the above equation we obtain
\begin{equation} \label{eq:Um-metric-time}
h_t(X, X) = d \dot{\alpha}_t (X, J X).
\end{equation}

Recall from Definition~\ref{defn:structures} that the submanifold $M$ is \emph{almost complex} (or $J$-holomorphic) if its tangent spaces $T_x M$ are preserved by $J$, and that $J$ is not varying in time.

\begin{thm}[Theorem A for $\U{m}$-structures] \label{thm:Um-A}
Let $M$ be an almost complex submanifold of $\ol{M}$, with $\dim M = 2k$ where $1 \leq k \leq m-1$. For small $|t|$, let $\ol{g}_t$ be a family of metrics on $\ol{M}$ of the form $\ol{g}_t (X, Y) = \omega_t (X, JY)$ where $\omega_t = \omega + (d \alpha_t)^{(1,1)}$ and $\alpha_t$ is a family of compactly supported smooth $1$-forms on $\ol{M}$ with $\alpha_{t=0} = 0$.
\begin{enumerate}[(a)]
\item Suppose $k = 1$, with \emph{no conditions on the torsion} of the $\U{m}$-structure. Then the functional $\cV$ is actually \emph{constant} on the set of metrics $\ol{g}_t$ of this form. In particular, $\ol{g}$ is a critical point of $\cV$ with respect to variations in this class.
\item Suppose $k > 1$, with \emph{no conditions on the torsion} of the $\U{m}$-structure, and suppose further that
\begin{equation} \label{eq:omegakdot-exact}
\rest{(d \dot{\alpha}_t \w \omega^{k-1}_t)}{M} \quad \text{is \emph{exact} at time $t$. (Note that this condition \emph{depends} on $M$.)}
\end{equation}
Then $\ddt \cV(\ol{g}_t) = 0$ at this time $t$. In particular, if~\eqref{eq:omegakdot-exact} holds at $t=0$, then $\ol{g}$ is a critical point of $\cV$ with respect to variations in this class.
\item Suppose $k > 1$ and $d \omega = 0$. Then $\rest{\ddt}{t=0} \cV(\ol{g}_t) = 0$. In particular, $\ol{g}$ is a critical point of $\cV$ with respect to variations in this class.
\end{enumerate}
\end{thm}
\begin{proof}
Fix a time $t$ in the domain of the family $\ol{g}_t$. Let $g_t$ and $h_t$ be as in Lemma~\ref{lemma:first-variation}. Since $M$ is almost complex, and $J$ is orthogonal with respect to $\ol{g}_t$, we can choose a local adapted frame for $(\ol{M}, \ol{g}_t)$ of the form
$$\{ e_1, J e_1, \ldots, e_k, J e_k, e_{k+1}, J e_{k+1}, \ldots, e_m, J e_m \} $$
so that, along $M$, the first $2k$ elements are an oriented orthonormal frame for $(M, g_t)$ and the last $2(m-k)$ elements are an oriented orthonormal frame for the normal bundle of $M$ in $\ol{M}$ with respect to $\ol{g}_t$. We need to compute the integrand $(\tr_{g_t} h_t)  \vol_{(M, g_t)} = (g_t)^{ab} (h_t)_{ab}  \vol_{(M, g_t)}$ of~\eqref{eq:first-variation}. Using~\eqref{eq:Um-metric-time} and our adapted frame we compute
\begin{equation} \label{eq:Um-A-g}
\begin{aligned}
(g_t)^{ab} (h_t)_{ab} & = \sum_{a=1}^k h_t(e_a, e_a) + \sum_{a=1}^k h_t(J e_a, J e_a) = \sum_{a=1}^k d \dot{\alpha}_t (e_a, J e_a) + \sum_{a=1}^k d \dot{\alpha}_t (J e_a, J^2 e_a) \\
& = \sum_{a=1}^k d \dot{\alpha}_t (e_a, J e_a) - \sum_{a=1}^k d \dot{\alpha}_t (J e_a, e_a) = 2 \sum_{a=1}^k d \dot{\alpha}_t (e_a, J e_a).
\end{aligned}
\end{equation}
To continue, let $\iota \colon M \to \ol{M}$ denote the inclusion, which is a smooth embedding. We have
\begin{equation} \label{eq:Um-A-tempb}
\begin{aligned}
\iota^* d \dot{\alpha}_t & = \frac{1}{2} \sum_{a,b=1}^k d \dot{\alpha}_t (e_a, e_b) e_a \w e_b + \frac{1}{2} \sum_{a,b=1}^k d \dot{\alpha}_t (e_a, J e_b) e_a \w J e_b \\
& \qquad {} + \frac{1}{2} \sum_{a,b=1}^k d \dot{\alpha}_t (J e_a, e_b) J e_a \w e_b + \frac{1}{2} \sum_{a,b=1}^k d \dot{\alpha}_t (J e_a, J e_b) J e_a \w J e_b.
\end{aligned}
\end{equation}

When $k=1$, simplifying the expression~\eqref{eq:Um-A-tempb} and using~\eqref{eq:Um-A-g}, we obtain
\begin{equation} \label{eq:Um-A-tempc}
\iota^* d \dot{\alpha}_t = d \dot{\alpha}_t (e_1, J e_1) e_1 \w J e_1 = \frac{1}{2} \big( (g_t)^{ab} (h_t)_{ab} \big) \vol_{(M, g)}.
\end{equation}
Thus, by~\eqref{eq:first-variation} and Stokes's Theorem, we conclude that
$$\ddt \cV (\ol{g}_t) = \int_M \iota^* d \dot{\alpha}_t = 0 \quad \text{for all $t$},$$
so $\cV$ is \emph{constant} on the set of metrics $\ol{g}_t$ of the form~\eqref{eq:Um-metric} with $\omega_t$ given by~\eqref{eq:Um-omegat-2} and~\eqref{eq:Um-omegat}. In particular, $\ol{g}$ is a critical point of $\cV$. This establishes part (a).

Now consider the case $k > 1$. Since $(\ol{g}_t, J, \omega_t)$ is a $\U{m}$-structure on $\ol{M}$ and $M$ is an almost complex submanifold, with respect to our adapted frame we have $\omega_t = \sum_{a=1}^m e_a \w J e_a$. Then we have
$$ \iota^* \omega_t = \sum_{a=1}^k e_a \w J e_a, $$
and hence
\begin{equation} \label{eq:Um-A-temp}
\iota^* \Big( \frac{1}{(k-1)!} \omega_t^{k-1} \Big) = \sum_{a=1}^k (e_1 \w J e_1) \w \cdots \w \wh{(e_a \w J e_a)} \w \cdots \w (e_k \w J e_k).
\end{equation}
We take the wedge product of~\eqref{eq:Um-A-tempb} with~\eqref{eq:Um-A-temp}. It is clear that only the second and third terms of~\eqref{eq:Um-A-tempb} will contribute, and after some simplification and using~\eqref{eq:Um-A-g}, we obtain
\begin{equation} \label{eq:Um-A-temp2}
\begin{aligned}
\iota^* \Big( d \dot{\alpha}_t \w \frac{1}{(k-1)!} \omega_t^{k-1} \Big) & = \Big( \sum_{a=1}^k d \dot{\alpha}_t (e_a, J e_a) \Big) e_1 \w J e_1 \w \cdots \w e_k \w J e_k \\
& = \Big( \sum_{a=1}^k d \dot{\alpha}_t (e_a, J e_a) \Big) \vol_{(M, g)} = \frac{1}{2} \big( (g_t)^{ab} (h_t)_{ab} \big) \vol_{(M, g)}.
\end{aligned}
\end{equation}
Suppose that condition~\eqref{eq:omegakdot-exact} holds, so $\frac{1}{(k-1)!} d \dot{\alpha}_t \w \omega^{k-1}_t$ is \emph{exact}. Then its pullback is also exact, and thus using~\eqref{eq:first-variation} and Stokes's Theorem, the above expression yields
\begin{equation}
\ddt \cV (\ol{g}_t) = 0 \qquad \text{at time $t$, if~\eqref{eq:omegakdot-exact} holds at this time.} \label{eq:Um-A-temp3}
\end{equation}
In particular, if~\eqref{eq:omegakdot-exact} holds at $t=0$, then $\ol{g}$ is a critical point of $\cV$. This establishes part (b).

In general, we can write
$$ d \dot{\alpha}_t \w \frac{1}{(k-1)!} \omega_t^{k-1} = d \Big( \dot{\alpha}_t \w \frac{1}{(k-1)!} \omega_t^{k-1} \Big) - \dot{\alpha}_t \w d \omega_t \w \Big( \frac{1}{(k-2)!} \omega_t^{k-2} \Big).$$
Substituting the above into~\eqref{eq:Um-A-temp2} and using~\eqref{eq:first-variation} and Stokes's Theorem, and writing $\rest{\gamma}{M}$ for $\iota^* \gamma$ for any form $\gamma$, we conclude that
\begin{equation} \label{eq:Um-A-temp4}
\ddt \cV (\ol{g}_t) = - \frac{1}{(k-2)!} \int_M (\rest{\dot{\alpha}_t}{M}) \w (d \rest{\omega_t}{M}) \w (\rest{\omega_t}{M})^{k-2}.
\end{equation}
If $d \omega = 0$, the right hand side above vanishes at $t=0$. Thus if $d\omega = 0$ then $\ol{g}$ is a critical point of $\cV$. This establishes part (c).
\end{proof}

We make some remarks about Theorem~\ref{thm:Um-A}.

\begin{rmk} \label{rmk:Um-A-a}
Part (b) says that if the family $\omega_t$ is such that the restriction of $d \dot \alpha_t \w \omega_t^{k-1}$ to $M$ is \emph{exact} for all $t$, then $\cV(\ol{g}_t)$ is \emph{constant}. This clearly always holds when $k=1$, consistent with part (a).

Part (a) was proved by Arezzo--Sun~\cite[Proposition 2.3]{AZ1}, assuming that $d \omega = 0$, but we have shown that this assumption is not necessary.
\end{rmk}

\begin{rmk} \label{rmk:Um-A-b}
For part (c), it is evident from~\eqref{eq:Um-A-temp4} that the seemingly weaker condition that $d \omega^{k-1} = 0$ suffices. However, since $k < m$, we have $k-1 < m-1$, and it follows~\cite[Proposition 1.2.30, (iv)]{Huybrechts} from the Lefschetz decomposition of $d \omega$ that $d \omega^j = 0 \iff d \omega = 0$ when $j < m-1$, so the condition $d \omega^{k-1} = 0$ is not actually weaker.

Nevertheless, the condition $d \omega = 0$ of part (c) is sufficient but not necessary. Let $d \rest{\omega_t}{M} = \theta_t \w \rest{\omega_t}{M} + \zeta_t$ be the Lefschetz decomposition of $d \rest{\omega_t}{M}$, where $\theta_t$ is the \emph{Lee form} of the almost Hermitian manifold $(M, J, \rest{\ol{g}_t}{M}, \rest{\omega_t}{M})$, and $\zeta_t$ is a primitive $3$-form. Then it is not hard to verify that $\ddt \cV (\ol{g}_t) = 0$ holds at time $t$, provided that $\dot{\alpha}_t$ is $g_t$ orthogonal to $\theta_t$. In particular, if we want this to hold at time $t=0$ for all $\dot{\alpha}_{t=0}$, then we must have $\theta_{t=0} = 0$. Thus in general, for $k > 1$, if we want $\ol{g}$ to be a critical point of $\cV$ when $M$ is almost complex, then the space of allowable metric variations~\eqref{eq:Um-metric} with $\omega_t$ given by~\eqref{eq:Um-omegat-2} and~\eqref{eq:Um-omegat} \emph{depends} on the submanifold $M$ of $\ol{M}$. (Compare with Remark~\ref{rmk:Cayley-cannot-fix}.)
\end{rmk}

The ``converse'' (see Remark~\ref{rmk:B-stronger}) of Theorem~\ref{thm:Um-A} is much simpler to state, as there are no subcases.

\begin{thm}[Theorem B for $\U{m}$-structures] \label{thm:Um-B}
Suppose that $\ol{g}$ is a critical point of the functional $\cV$ with respect to variations of the form $\ol{g}_t (X, Y) = \omega_t (X, JY)$ where $\omega_t = \omega + (d \alpha_t)^{(1,1)}$ and $\alpha_t$ is a family of compactly supported smooth $1$-forms on $\ol{M}$ with $\alpha_{t=0} = 0$. Then the submanifold $M$ is almost complex.
\end{thm}
\begin{proof}
In this proof, the ranges for our indices are $1 \leq a, b \leq 2k$ and $1 \leq i, j, l, p, q \leq 2m$. That is, $\{ e_1, \ldots, e_{2m} \}$ is a local orthonormal frame for $\ol{M}$ adapted to the submanifold $M$.

Given any compactly supported smooth $1$-form $\dot\alpha$, let $\ol{g}_t$ be the variation obtained as above by taking $\alpha_t = t \dot \alpha$ for $|t|$ small. Using~\eqref{eq:first-variation}, our hypothesis then gives that
\begin{equation} \label{eq:Um-B-hypothesis}
\rest{\ddt}{t=0} \cV (\ol{g}_t) = \frac{1}{2} \int_M (\tr_g h) \, \vol_{(M, g)} = 0,
\end{equation}
where $h$ is given in terms of $\dot \alpha$ by~\eqref{eq:Um-metric-time}. That is, we have
\begin{equation} \label{eq:Um-B-start}
h_{ii} = J^q_i (d \dot \alpha)_{iq} \quad \text{(no sum on $i$.)}
\end{equation}
Taking the trace of $h$ in the tangential directions gives
\begin{align}
\tr_g h = \sum_{a=1}^{2k} h_{aa} & = \sum_{a=1}^{2k} J^q_a (d \dot \alpha)_{aq}. \label{eq:Um-h-trace}
\end{align}

We need to make an appropriate choice of $\dot \alpha$. We choose
\begin{equation} \label{eq:Um-test-variation}
\dot \alpha = J (\ol{\nabla} F) = \ol{\nabla} F \hk \omega
\end{equation}
where $F$ is as in Definition~\ref{defn:F}. (This is sometimes called the \emph{symplectic gradient} of $F$, but recall that we do not assume that $\omega$ is closed.) In terms of a local frame, we have $\dot \alpha_j = J^p_j \ol{\nabla}_p F$, and thus
\begin{align*}
(d \dot \alpha)_{ij} = \ol{\nabla}_i \dot \alpha_j - \ol{\nabla}_j \dot \alpha_i & = J^p_j \ol{\nabla}_i \ol{\nabla}_p F - J^p_i \ol{\nabla}_j \ol{\nabla}_p F + \lf \\
& = J^p_j H_{ip} - J^p_i H_{jp} + \lf.
\end{align*}
Recall that on $M$, the $\lf$ terms vanish. Thus, using $J_a^q J_q^p = - \delta_a^p$ and the above, along the submanifold $M$ equation~\eqref{eq:Um-h-trace} becomes
$$ \tr_g h = \sum_{a=1}^{2k} J^q_a (J^p_q H_{ap} - J^p_a H_{qp}) = \sum_{a=1}^{2k} ( - H_{aa} - J^p_a J^q_a H_{pq}  ). $$
Since $e_a$ is tangent along $M$ for $1 \leq a \leq 2k$, by~\eqref{eq:Hessian-F}, we have
$$ H_{aa} = | e_a^{\perp} |^2 = 0 \quad \text{and} \quad J^p_a J^q_a H_{pq} = H(Je_a, Je_a) = | (J e_a)^{\perp} |^2, $$ 
and thus
$$ \tr_g h = - \sum_{a=1}^{2k} | (Je_a)^{\perp} |^2. $$

Hence from the above and~\eqref{eq:Um-B-hypothesis}, we finally obtain
$$ 0 = \rest{\ddt}{t=0} \cV(\ol{g}_t) = - \frac{1}{2} \int_M \left( \sum_{a=1}^{2k} \big| (J e_a)^{\perp} \big|^2 \right) \vol_{(M, g)}. $$
The integrand must vanish as it is nonnegative. Hence $(J e_a)^{\perp} = 0$ for $1 \leq a \leq 2k$. Thus $J e_a$ is tangent to $M$ for $1 \leq a  \leq 2k$ and we conclude that $M$ is almost complex.
\end{proof}

\begin{rmk} \label{rmk:Um-B}
Theorem~\ref{thm:Um-B} was proved by Arezzo--Sun~\cite[Theorem 1.2]{AZ2}, assuming that $d \omega = 0$, but we have shown that this assumption is not necessary. See Remark~\ref{rmk:Um-A-a}.
\end{rmk}

It is worthwhile to be more precise about the family of variations $\wt \omega_t$ considered in the work of Arezzo--Sun in~\cite{AZ1, AZ2}, to compare and contrast with our choices of variations for the $\G$ and $\Spin{7}$ calibrations later in Sections~\ref{sec:G2} and~\ref{sec:Spin7}. Arezzo--Sun considered variations of the form $\wt \omega_t = \omega + d d^c f_t$ for some one-parameter family of real functions $f_t$, where $d^c f_t = J^{-1} d (J f_t) = - J d f_t = - \ol{\nabla} f_t \hk \omega$.

On functions we always have $d^c = \sqrt{-1} (\delbar - \del)$, even if $J$ is not integrable. Thus we have $d d^c f_t = \sqrt{-1} d (\delbar f_t - \del f_t)$. If $J$ is assumed integrable, then $d = \del + \delbar$, so $d d^c f_t = 2 \sqrt{-1} \del \delbar f_t$ in this case. Independently, if $d \omega = 0$, then one can compute that $d d^* (f_t \omega) = d d^c f_t$. In fact this calculation only requires that the Lee form of $(\ol{M}, \ol{g}, J, \omega)$ vanishes, which is equivalent to $d^* \omega = 0$ or $\omega^{m-2} \w d \omega = 0$.

In summary, the one-parameter family $\alpha_t$ of Arezzo--Sun satisfied
\begin{equation} \label{eq:AZ}
d \alpha_t = d d^c f_t = \begin{cases} - d (\ol{\nabla} f_t \hk \omega)& \text{(always)}, \\
2 \sqrt{-1} \del \delbar f_t & \text{(if $J$ is integrable)}, \\
dd^* (f_t \omega) & \text{(if $\omega^{m-2} \w d \omega = 0$)}. \end{cases}
\end{equation}
In particular, their variations could equivalently have been stated to be of the form $d \alpha_t = d d^* (f_t \omega)$.

Recall from~\eqref{eq:Um-test-variation} that our test variation to prove Theorem~\ref{thm:Um-B} satisfied $d \dot \alpha = d (\ol{\nabla} F \hk \omega)$, which is exactly of the form $d d^c \dot f = - d( \ol{\nabla} \dot f \hk \omega)$ of Arezzo--Sun, with $\dot f = - F$. See the discussion in Remark~\ref{rmk:G2-attempts} below regarding our initial attempts to find the correct analogue of these variations $\alpha_t$ in order to successfully generalize Theorem B to $\G$-structures.

\subsection{Submanifolds of manifolds with $\G$-structure} \label{sec:G2}

In this section we consider the case when $\ol{M}$ has a $\G$-structure.

\subsubsection{Generalities on manifolds with $\G$-structure and their distinguished submanifolds} \label{sec:G2-general}

Suppose that $n=7$ and let $(\ph, \ol{g})$ be a $\G$-structure on $\ol{M}$. Here $\ph$ is a nondegenerate $3$-form which determines the Riemannian metric $\ol{g}$ in a nonlinear way. It also induces an orientation, and thus we have a Hodge dual $4$-form $\ps = \sta \ph$. See~\cites{Bryant, K1, K-flows, K-intro} for more about $\G$-structures. We do not assume anything about the \emph{torsion} $\ol{\nabla} \ph$ of the $\G$-structure.

There is a $2$-fold vector cross product $\times$ on $\ol{M}$ defined by
$$ \ph(X, Y, Z) = \ol{g}(X \times Y, Z).$$
That is, $\times$ is obtained from $\ph$ by raising an index using $\ol{g}$. It is clear that $X \times Y$ is orthogonal to each of $X, Y$. Moreover, $| X \times Y |^2 = | X \w Y |^2$.

Recall from Definition~\ref{defn:structures} that a $3$-dimensional orientable submanifold $M$ of $(\ol{M}, \ol{g}, \ph, \ps)$ is \emph{associative} if its tangent spaces $T_x M$ are preserved by the cross product --- that is, if $X \times Y$ is tangent to $M$ whenever $X, Y$ are both tangent to $M$. This is equivalent to the existence of an orientation on $M$ with respect to which $\rest{\ph}{M} = \vol_M$. To see this, let $X, Y, Z$ be orthonormal tangent vectors to $M$ and observe by Cauchy--Schwarz that
$$ | \ph(X, Y, Z) | = | \ol{g} (X \times Y, Z) | \leq | X \times Y | \, |Z| = | X \w Y | \, | Z | = 1 = | \vol_M (X, Y, Z) |, $$
with equality if and only if $X \times Y = \pm Z$. Note that when $d \ph = 0$, an associative submanifold is \emph{calibrated} by $\ph$ and hence minimal~\cite{HL}, but we do not assume $d \ph = 0$.

Similarly, there is a vector-valued $3$-form $\chi$ on $\ol{M}$ defined by
$$ \ps(X, Y, Z, W) = \ol{g}(\chi(X, Y, Z), W).$$
That is, $\chi$ is obtained from $\ps$ by raising an index using $\ol{g}$. While it is clear that $\chi(X, Y, Z)$ is orthogonal to each of $X, Y, Z$, the form $\chi$ is \emph{not} a vector cross product, because $| \chi(X, Y, Z) |^2$ \emph{does not} equal $|X \w Y \w Z|^2$. In fact, the failure of $\chi$ to be a vector cross product is precisely measured by the ``associative equality''
\begin{equation} \label{eq:assoc-equality}
| \chi(X, Y, Z) |^2 + \ph(X, Y, Z)^2 = |X \w Y \w Z|^2.
\end{equation}
Another important identity is the ``coassociative equality''
\begin{equation} \label{eq:coassoc-equality}
\ps(X, Y, Z, W)^2 + | \ph(Y, Z, W) X - \ph(X, Z, W) Y + \ph(X, Y, W) Z - \ph(X, Y, Z) W |^2 = | X \w Y \w Z \w W |^2.
\end{equation}
(See~\cite[Theorem 1.6, p.\ 114]{HL} for a proof of~\eqref{eq:assoc-equality} and~\cite[Theorem 1.18, p.\ 117]{HL} for a proof of~\eqref{eq:coassoc-equality}.) We remark in passing that~\eqref{eq:assoc-equality} shows that a $3$-dimensional submanifold $M$ is associative if and only if $\chi$ restricts to zero on $M$.

Recall from Definition~\ref{defn:structures} that a $4$-dimensional orientable submanifold $M$ of $(\ol{M}, \ol{g}, \ph, \ps)$ is \emph{coassociative} if $\rest{\ph}{M} = 0$. This is equivalent to the existence of an orientation on $M$ with respect to which $\rest{\ps}{M} = \vol_M$. To see this, let $X, Y, Z, W$ be orthonormal tangent vectors to $M$ and observe from~\eqref{eq:coassoc-equality} that
$$ \rest{\ph}{M} = 0 \, \, \text{ if and only if } \, \, | \ps(X, Y, Z, W) | = 1 = | \vol_M (X, Y, Z, W) |. $$
Note that when $d \ps = 0$, a coassociative submanifold is \emph{calibrated} by $\ps$ and hence minimal~\cite{HL}. But we do not assume $d \ps = 0$.

The next result is likely well-known to experts but we were unable to locate it in the literature. It is needed in the proof of Theorem~\ref{thm:coassoc-B}.

\begin{prop} \label{prop:coassoc-condition}
Let $M$ be a $4$-dimensional submanifold of $\ol{M}$. Then $M$ is coassociative if and only if $\chi(X, Y, Z)$ is tangent to $M$ whenever $X, Y, Z$ are all tangent to $M$.
\end{prop}
\begin{proof}
Suppose that $M$ is coassociative. Let $\{e_0, e_1, e_2, e_3\}$ be an orthonormal basis of $T_x M$, and let $(i,j,k,l)$ be any permutation of $(0,1,2,3)$. Using that $M$ is coassociative, the Cauchy--Schwarz inequality, and~\eqref{eq:assoc-equality}, we have
$$ 1 = | \ps(e_i, e_j, e_k, e_l) | = | \ol{g}( \chi (e_i, e_j, e_k), e_l) | \leq | \chi (e_i, e_j, e_k) | | e_l | \leq | e_i \w e_j \w e_k | = 1 .$$
Thus, equality holds in Cauchy--Schwarz, so $\{ \chi (e_i, e_j, e_k ), e_l \}$ is linearly dependent. We deduce that $\chi (e_i, e_j, e_k) \in T_x M$, and hence $\chi (X, Y, Z)$ is tangent to $M$ whenever $X, Y, Z$ are tangent to $M$.

Conversely, suppose that $\chi(X, Y, Z)$ is tangent to $M$ whenever $X, Y, Z$ are tangent to $M$. Let $\{ e_0, e_1, e_2, e_3 \}$ be an orthonormal basis of $T_x M$. Let $\lambda = \ps(e_0, e_1, e_2, e_3)$. We need to show that $\lambda = \pm 1$. Since $\chi(X, Y, Z) = \ps(X, Y, Z, \cdot)$ is tangent to $M$ and orthogonal to each of $X, Y, Z$, we have
$$ \chi(e_i, e_j, e_k) = \pm \lambda e_l \qquad \text{for $i, j, k, l$ distinct in $\{ 0, 1, 2, 3\}$}. $$
Substituting the above into~\eqref{eq:assoc-equality} gives
$$ \lambda^2 + \ph(e_i, e_j, e_k)^2 = 1 \qquad \text{for $i, j, k$ distinct in $\{ 0, 1, 2, 3\}$}. $$
Substituting $\ph(e_i, e_j, e_k)^2 = 1 - \lambda^2$ into~\eqref{eq:coassoc-equality} and using orthonormality of $\{ e_0, e_1, e_2, e_3\}$ gives
$$ \lambda^2 + 4(1 - \lambda^2) = 1, $$
and hence $\lambda = \pm 1$ as required.
\end{proof}

\textbf{Note.} Throughout the rest of Section~\ref{sec:G2}, the ranges for our indices are $1 \leq i, j, k, l, m, p, q \leq 7$. Moreover, $\{ e_1, \ldots, e_7 \}$ is a local orthonormal frame for $\ol{M}$, so $\ol{g}_{ij} = \delta_{ij}$. Thus all our indices will be subscripts and we sum over repeated indices. In \S\ref{sec:assoc} we have $1 \leq a, b \leq 3$, and in \S\ref{sec:coassoc} we have $4 \leq a, b \leq 7$ as the ranges for the indices which are tangent to the submanifold $M$. We explicitly display the summation symbols and ranges in equations that feature sums over different ranges, in the proofs of Theorems~\ref{thm:assoc-B} and~\ref{thm:coassoc-B}.

We need to make use of various contraction identities for $\G$-structures, whose proofs can be found in~\cite{K-flows}. The identities we need are:
\begin{equation} \label{eq:G2-identities}
\begin{aligned}
\ph_{ijp} \ph_{klp} & = \ol{g}_{ik} \ol{g}_{jl} - \ol{g}_{il} \ol{g}_{jk} - \ps_{ijkl}, \\
\ph_{ipq} \ph_{jpq} & = 6 \ol{g}_{ij}, \\
\ph_{ipq} \ps_{jkpq} & = - 4 \ph_{ijk}, \\
\ph_{mpq} \ps_{jmpq} & = 0, \\
\ps_{ijpq} \ps_{klpq} & = 4 \ol{g}_{ik} \ol{g}_{jl} - 4 \ol{g}_{il} \ol{g}_{jk} - 2 \ps_{ijkl}, \\
\ps_{impq} \ps_{jmpq} & = 24 \ol{g}_{ij}.
\end{aligned}
\end{equation}

\subsubsection{Associative submanifolds of manifolds with $\G$-structure} \label{sec:assoc}

Let $k=3$, so that $M$ is a compact oriented embedded $3$-dimensional submanifold of $\ol{M}$. We consider the functional $\cV$ of~\eqref{eq:functional} in this setting, and consider the following restricted class of variations of the ambient metric $\ol{g}$.

Let $\beta_t$ be a family of compactly supported smooth $2$-forms on $\ol{M}$ with $\beta_{t=0} = 0$. Then $\ph_t = \ph + d \beta_t$ is a one-parameter family of smooth $3$-forms on $\ol{M}$ with $\ph_{t=0} = \ph$. Note that if $d \ph = 0$, then $\ph_t$ lies in the same cohomology class as $\ph$ for all $t$, but we do not assume that $d \ph = 0$.

The space of nondegenerate $3$-forms on $\ol{M}$ is an open subbundle of the bundle $\Omega^3$ of smooth $3$-forms. Hence, for small $|t|$, the $3$-form $\ph_t$ is still nondegenerate. Consider the Riemannian metric $\ol{g}_t$ on $\ol{M}$ where
\begin{equation} \label{eq:assoc-metric}
\text{$\ol{g}_t$ is the Riemannian metric induced by $\ph_t = \ph + d \beta_t$}.
\end{equation}
Write $\dot{\beta}$ for $\rest{\ddt}{t=0} \beta_t$. We therefore have
\begin{equation} \label{eq:assoc-form-time}
\rest{\ddt}{t=0} \ph_t = d \dot{\beta}.
\end{equation}

The precise nonlinear formula for $\ol{g}_t$ in terms of $\ph_t = \ph + d \beta_t$ is quite complicated. For our purposes, all that matters is the first variation $\rest{\ddt}{t=0} \ol{g}_t$. We now discuss how this is determined. Using the initial $\G$-structure $\ph$, the space $\Omega^3$ of smooth $3$-forms on $\ol{M}$ decomposes orthogonally with respect to $\ol{g}$ into
$$ \Omega^3 = \Omega^3_1 \oplus \Omega^3_{27} \oplus \Omega^3_7.$$
A $3$-form $\eta$ on $\ol{M}$ therefore decomposes as $\eta = \eta_{1+27} + \eta_7$, where~\cite{K-flows, K-intro} we have
\begin{equation} \label{eq:G2-3-form-decomp}
\eta_{1+27} = \frac{1}{2} h_{ij} e_i \w (e_j \hk \ph), \qquad \eta_7 = \frac{1}{2} X \hk \ps,
\end{equation}
for some unique vector field $X$ and unique symmetric $2$-tensor $h$ on $\ol{M}$. (The factor of $\frac{1}{2}$ is chosen in~\eqref{eq:G2-3-form-decomp} to eliminate a factor of $2$ in~\eqref{eq:assoc-metric-time} below.) Note that we can decompose $h = \frac{1}{7} (\tr_{\ol{g}} h) \ol{g} + h^0$ into its pure-trace and trace-free parts, which are orthogonal with respect to $\ol{g}$, and which correspond to the $\Omega^3_1$ and $\Omega^3_{27}$ components of $\eta$, respectively. But for the consideration of the first variation of the metric, it is natural to combine these two components as we have in~\eqref{eq:G2-3-form-decomp} above. It is shown in~\cite{Bryant, K-flows} that
\begin{equation} \label{eq:assoc-metric-time}
\rest{\ddt}{t=0} \ph_t = \eta \quad \implies \quad \rest{\ddt}{t=0} \ol{g}_t = h,
\end{equation}
where $h$ is related to $\eta$ by~\eqref{eq:G2-3-form-decomp}.

\begin{lemma} \label{lemma:h-from-eta}
Let $\eta$, $X$, and $h$ be as in~\eqref{eq:G2-3-form-decomp}. Define the $2$-tensor $\wh{\eta}$ on $\ol{M}$ by $\wh{\eta}_{pq} = \eta_{pij} \ph_{qij}$. Then we have
\begin{equation} \label{eq:h-from-eta}
h_{il} = \frac{1}{4} ( \wh{\eta}_{il} + \wh{\eta}_{li} ) - \frac{1}{18} (\tr_{\ol{g}} \wh{\eta}) \ol{g}_{il}.
\end{equation}
\end{lemma}
\begin{proof}
From~\eqref{eq:G2-3-form-decomp}, we have
\begin{equation} \label{eq:eta-temp}
2 \eta_{ijk} = h_{ip} \ph_{pjk} + h_{jp} \ph_{ipk} + h_{kp} \ph_{ijp} + X_p \ps_{pijk}.
\end{equation}
Using~\eqref{eq:eta-temp}, the symmetry of $h$, and the identities in~\eqref{eq:G2-identities}, we compute
\begin{align*}
2 \wh{\eta}_{il} & = (h_{ip} \ph_{pjk} + h_{jp} \ph_{ipk} + h_{kp} \ph_{ijp} + X_p \ps_{pijk}) \ph_{ljk} \\
& = h_{ip} (6 \ol{g}_{pl}) + 2 h_{jp} (\ol{g}_{il} \ol{g}_{pj} - \ol{g}_{ij} \ol{g}_{pl} - \ps_{iplj}) + X_p (-4 \ph_{lpi}) \\
& = 4 h_{il} + 2 (\tr_{\ol{g}} h) \ol{g}_{il} - 4 X_p \ph_{pil}. 
\end{align*}
We deduce that $2 \tr_{\ol{g}} \wh{\eta} = 18 \tr_{\ol{g}} h$ and $2 \wh{\eta}_{il} + 2 \wh{\eta}_{li} = 8 h_{il} + 4 (\tr_{\ol{g}} h) \ol{g}_{il}$. Equation~\eqref{eq:h-from-eta} follows.
\end{proof}

\begin{thm}[Theorem A for associative submanifolds] \label{thm:assoc-A}
Let $M$ be an associative submanifold of $\ol{M}$. For small $|t|$, let $\ol{g}_t$ be a family of metrics on $\ol{M}$ induced by $\ph_t = \ph + d \beta_t$ with $\beta_{t = 0} = 0$. Then $\rest{\ddt}{t=0} \cV(\ol{g}_t) = 0$, so $\ol{g}$ is a critical point of $\cV$ with respect to variations in this class.
\end{thm}
\begin{proof}
Since $M$ is associative, we can choose a local adapted frame $\{ e_1, \ldots, e_7 \}$ for $(\ol{M}, \ol{g})$ so that, along $M$, the first $3$ elements are an oriented local orthonormal frame for $(M, g)$ and the last $4$ elements are an oriented local orthonormal frame for the normal bundle of $M$ in $\ol{M}$ with respect to $\ol{g}$, and moreover that the frame is $\G$-adapted in the sense that
$$ e_3 = e_1 \times e_2, \quad e_5 = e_1 \times e_4, \quad e_6 = e_2 \times e_4, \quad e_7 = (e_1 \times e_2) \times e_4.$$
With respect to such a frame, along $M$ we have
\begin{equation} \label{eq:assoc-A-phi}
\begin{aligned}
\ph & = e_1 \w e_2 \w e_3 + e_1 \w (e_4 \w e_5 - e_6 \w e_7) \\ & \qquad {} + e_2 \w (e_4 \w e_6 - e_7 \w e_5) + e_3 \w (e_4 \w e_7 - e_5 \w e_6).
\end{aligned}
\end{equation}
Write $\eta$ for $\rest{\ddt}{t=0} \ph_t = d \dot{\beta}$ and $h$ for $\rest{\ddt}{t=0} \ol{g}_t$. Then $h$ is given in terms of $\eta$ by~\eqref{eq:h-from-eta}, and we need to compute the integrand $(\tr_{g} h)  \vol_{(M, g)} = g^{ab} h_{ab}  \vol_{(M, g)}$ of~\eqref{eq:first-variation} (at $t=0$). With $\wh{\eta}$ defined as in Lemma~\ref{lemma:h-from-eta}, we compute with respect to this local frame that
\begin{equation} \label{eq:trace-hat-eta}
\tr_{\ol{g}} \wh{\eta} = \eta_{ijk} \ph_{ijk} = 6 (\eta_{123} + \eta_{145} - \eta_{167} + \eta_{246} - \eta_{275} + \eta_{347} - \eta_{356}),
\end{equation}
where the factor of $6 = 3!$ arises because we are summing over all $1 \leq i, j, k \leq 7$. Similarly, with respect to this local frame, we need to compute
\begin{equation} \label{eq:hat-eta-temp}
\wh{\eta}_{aa} = \eta_{aij} \ph_{aij} = \eta_{1ij} \ph_{1ij} + \eta_{2ij} \ph_{2ij} + \eta_{3ij} \ph_{3ij}.
\end{equation}
For example, using~\eqref{eq:assoc-A-phi} the term $\eta_{1ij} \ph_{1ij}$ is $2 ( \eta_{123} + \eta_{145} - \eta_{167} )$. Combining the terms of~\eqref{eq:hat-eta-temp} and using~\eqref{eq:trace-hat-eta} we obtain
\begin{align*}
\wh{\eta}_{aa} & = 2 ( 3 \eta_{123} + \eta_{145} - \eta_{167} + \eta_{246} - \eta_{275} + \eta_{347} - \eta_{356} ) \\
& = 4 \eta_{123} + 2 (\eta_{123} + \eta_{145} - \eta_{167} + \eta_{246} - \eta_{275} + \eta_{347} - \eta_{356}) \\
& = 4 \eta_{123} + \frac{2}{6} \tr_{\ol{g}} \wh{\eta}.
\end{align*}
Using~\eqref{eq:h-from-eta} and the above, we compute
$$ h_{aa} = \frac{1}{2} \wh{\eta}_{aa} - \frac{3}{18} \tr_{\ol{g}} \wh \eta = \frac{1}{2} (4 \eta_{123} + \frac{2}{6} \tr_{\ol{g}} \wh{\eta}) - \frac{1}{6} \tr_{\ol{g}} \wh \eta = 2 \eta_{123}.$$
Hence, we have $(\tr_{g} h)  \vol_{(M, g)} = h_{aa} e_1 \w e_2 \w e_3 = 2 \eta_{123} e_1 \w e_2 \w e_3 = 2 \rest{\eta}{M}$. Recalling that $\eta = d \dot{\beta}$, by Stokes's Theorem and~\eqref{eq:first-variation}, we deduce that
$$ \rest{\ddt}{t=0} \cV(\ol{g}_t) = \int_M \rest{(d \dot{\beta})}{M} = 0, $$
completing the proof.
\end{proof}

\begin{thm}[Theorem B for associative submanifolds] \label{thm:assoc-B}
Suppose that $\ol{g}$ is a critical point of the functional $\cV$ with respect to variations $\ol{g}_t$ induced by $\ph_t = \ph + d \beta_t$. Then the submanifold $M$ is associative.
\end{thm}
\begin{proof}
Given any compactly supported smooth $2$-form $\dot\beta$, let $\ol{g}_t$ be the variation obtained as above by taking $\beta_t = t \dot \beta$ for $|t|$ small, and write $\eta$ for $d\dot\beta = \rest{\ddt}{t=0} \ph_t$. Using~\eqref{eq:first-variation}, our hypothesis then gives that
\begin{equation} \label{eq:assoc-B-hypothesis}
\rest{\ddt}{t=0} \cV (\ol{g}_t) = \frac{1}{2} \int_M (\tr_g h) \, \vol_{(M, g)} = 0,
\end{equation}
where by~\eqref{eq:assoc-metric-time} and Lemma~\ref{lemma:h-from-eta}, the symmetric 2-tensor $h$ is given by
\begin{equation} \label{eq:assoc-B-start}
h_{il} = \frac{1}{4}( \wh{\eta}_{il} + \wh{\eta}_{li} ) - \frac{1}{18} (\tr_{\ol{g}} \wh{\eta}) \ol{g}_{il}, \qquad \text{where} \qquad \wh{\eta}_{il} = \eta_{imn} \ph_{lmn}.
\end{equation}
Taking the trace of $h$ in the tangential directions gives
\begin{align}
\tr_g h = \sum_{a=1}^3 h_{aa} & = \frac{1}{2} \sum_{a=1}^3 \wh{\eta}_{aa} - \frac{1}{6} \sum_{i=1}^7 \wh{\eta}_{ii} \nonumber \\
& = \frac{1}{2} \sum_{a=1}^3 \eta_{amn} \ph_{mna} - \frac{1}{6} \sum_{i=1}^7 \eta_{imn} \ph_{mni} \nonumber \\
& = \frac{1}{2} \eta((e_m \times e_n)^T, e_m, e_n) - \frac{1}{6} \eta( e_m \times e_n, e_m, e_n ) \nonumber \\
& = \frac{1}{6} \eta \big( 2 (e_m \times e_n)^T - (e_m \times e_n)^{\perp}, e_m, e_n \big). \label{eq:assoc-h-trace}
\end{align}

We need to make an appropriate choice of $\dot \beta$. Let $V$ be any smooth tangent vector field on $M$ and extend it smoothly to a global vector field with compact support, still denoted by $V$, on the ambient space $\ol{M}$. We choose
\begin{equation} \label{eq:assoc-test-variation}
\dot \beta = V \w (V \times \ol{\nabla} F)
\end{equation}
where $F$ is as in Definition~\ref{defn:F}. In terms of a local frame, we have $(V \times \ol{\nabla} F)_j = V_p (\ol{\nabla}_q F) \ph_{pqj}$, and thus using the skew-symmetry of $\ph(X, Y, Z) = \langle X \times Y, Z \rangle$, we have
\begin{align} \nonumber
\dot \beta_{ij} & = V_i \langle V \times \ol{\nabla} F, e_j \rangle - V_j \langle V \times \ol{\nabla} F,  e_i \rangle \\
& = V_j \langle V \times e_i, \ol{\nabla} F \rangle - V_i \langle V \times e_j, \ol{\nabla} F \rangle. \label{eq:beta-coords}
\end{align}
Recall from Lemma~\ref{lemma:distance-function} that, \emph{along the submanifold $M$}, we have $\ol{\nabla} F = 0$ and $\ol{\nabla}_X (\ol{\nabla} F) = X^{\perp}$. Using these facts and equation~\eqref{eq:beta-coords}, along $M$ we have
\begin{align}
(d \dot \beta)_{ijk} & = \ol{\nabla}_i \dot \beta_{jk} + \ol{\nabla}_j \dot \beta_{ki} + \ol{\nabla}_k \dot \beta_{ij} \nonumber \\
& = V_k \langle V \times e_j, e_i^{\perp} \rangle - V_j \langle V \times e_k, e_i^{\perp} \rangle \nonumber \\
& \qquad {} + V_i \langle V \times e_k, e_j^{\perp} \rangle - V_k \langle V \times e_i, e_j^{\perp} \rangle \nonumber \\
& \qquad {} + V_j \langle V \times e_i, e_k^{\perp} \rangle - V_i \langle V \times e_j, e_k^{\perp} \rangle. \label{eq:d-beta}
\end{align}

We need to compute the two terms on the right hand side of~\eqref{eq:assoc-h-trace} for $\eta = d \dot \beta$. Recall that $V = V_k e_k$ is tangent along $M$. Substituting $e_i = (e_j \times e_k)^T$ into~\eqref{eq:d-beta} and summing over $j, k = 1, \ldots, 7$, with some obvious manipulation we obtain
\begin{align*}
(d \dot \beta) \big( (e_j \times e_k)^T, e_j, e_k \big) & = 0 - 0 + \langle V, e_j \times e_k \rangle \langle V \times e_k, e_j^{\perp} \rangle - V_k \langle V \times (e_j \times e_k)^T, e_j^{\perp} \rangle \\
& \qquad {} + V_j \langle V \times (e_j \times e_k)^T, e_k^{\perp} \rangle - \langle V, e_j \times e_k \rangle \langle V \times e_j, e_k^{\perp} \rangle \\
& = - \langle V \times e_k, e_j \rangle \langle V \times e_k, e_j^{\perp} \rangle - \langle V \times e_j^{\perp}, (V \times e_j)^T \rangle \\
& \qquad {} - \langle V \times e_k^{\perp}, (V \times e_k)^T \rangle - \langle V \times e_j, (V \times e_j)^{\perp} \rangle \\
& = - 2 \langle (V \times e_k)^{\perp}, (V \times e_k)^{\perp} \rangle - 2 \langle V \times e_k^{\perp}, (V \times e_k)^T \rangle \\
& = - 2 \sum_{k=1}^3 |(V \times e_k)^{\perp}|^2 - 2 \sum_{k=4}^7 |(V \times e_k)^{\perp}|^2 - 2 \sum_{k=4}^7 |(V \times e_k)^T|^2.
\end{align*}
Noting that
$$ \sum_{k=4}^7 |(V \times e_k)^T|^2 = \sum_{a=1}^3 \sum_{k=4}^7 \langle V \times e_k, e_a \rangle^2 = \sum_{a=1}^3 \sum_{k=4}^7 \langle V \times e_a, e_k \rangle^2 = \sum_{a=1}^3 |(V \times e_a)^{\perp}|^2, $$
we thus obtain
\begin{equation} \label{eq:d-beta-T}
(d \dot \beta) \big( (e_j \times e_k)^T, e_j, e_k \big) = -4 \sum_{a=1}^3 |(V \times e_a)^{\perp}|^2 - 2 \sum_{k=4}^7 |(V \times e_k)^{\perp}|^2.
\end{equation}

Next, we substitute $e_i = (e_j \times e_k)^{\perp}$ into~\eqref{eq:d-beta} and sum over $j, k = 1, \ldots, 7$ to obtain
\begin{align*}
(d \dot \beta) \big( (e_j \times e_k)^{\perp}, e_j, e_k \big) & = V_k \langle V \times e_j, (e_j \times e_k)^{\perp} \rangle - V_j \langle V \times e_k, (e_j \times e_k)^{\perp} \rangle + 0 \\
& \qquad {} + V_k \langle V \times e_j^{\perp}, (e_j \times e_k)^{\perp} \rangle - V_j\langle V \times e_k^{\perp}, (e_j \times e_k)^{\perp} \rangle - 0 \\
& = -2 \langle (V \times e_j)^{\perp}, (V \times e_j)^{\perp} \rangle - 2 \langle V \times e_k^{\perp}, (V \times e_k)^{\perp} \rangle \\
& = - 2 \sum_{a=1}^3 |(V \times e_a)^{\perp}|^2 - 2 \sum_{j=4}^7 |(V \times e_j)^{\perp}|^2 - 2 \sum_{k=4}^7 |(V \times e_k)^{\perp}|^2,
\end{align*}
and hence
\begin{equation} \label{eq:d-beta-perp}
(d \dot \beta) \big( (e_j \times e_k)^{\perp}, e_j, e_k \big) = -2 \sum_{a=1}^3 |(V \times e_a)^{\perp}|^2 - 4 \sum_{k=4}^7 |(V \times e_k)^{\perp}|^2.
\end{equation}

Substituting the expressions~\eqref{eq:d-beta-T} and~\eqref{eq:d-beta-perp} into~\eqref{eq:assoc-h-trace} with $\eta = d \dot \beta$, we obtain
$$ \tr_g h = \frac{1}{6}(d \dot \beta) \big( 2 (e_m \times e_n)^T - (e_m \times e_n)^{\perp}, e_m, e_n \big) = - \sum_{a=1}^3 |(V \times e_a)^{\perp}|^2. $$
Hence from the above and~\eqref{eq:assoc-B-hypothesis}, we finally obtain
\begin{equation*}
0 = \rest{\ddt}{t=0} \cV(\ol{g}_t) = -\frac{1}{2} \int_M \left( \sum_{a=1}^{3} \big| (V \times e_a)^{\perp} \big|^2 \right) \vol_{(M, g)}.
\end{equation*}
The integrand must vanish as it is nonnegative. Hence $(V \times e_a)^{\perp} = 0$ for $a = 1, 2, 3$ and \emph{any tangent vector field $V$ on $M$}. Taking $V$ to be $e_1, e_2, e_3$ shows that $e_a \times e_b$ is tangent to $M$ for $1 \leq a, b \leq 3$ and we conclude that $M$ is associative.
\end{proof}

\begin{rmk} \label{rmk:G2-attempts}
The authors experienced several failed attempts to find the correct test variation~\eqref{eq:assoc-test-variation} that succeeds to prove Theorem~\ref{thm:assoc-B}. Comparing with $\dot \alpha = \ol{\nabla} F \hk \omega$ in the almost complex case and recalling equation~\eqref{eq:AZ}, the obvious first guess is to try $\dot \beta = \ol{\nabla} F \hk \ph$, which does not work. Similarly, again comparing with~\eqref{eq:AZ}, we considered taking $d \dot \beta = d d^* (F \ph)$ and possibly demanding that $d \ph = 0$ or even that $\ol{\nabla} \ph = 0$. That did not work either.

Next, the authors reasoned as follows. In moving from $\U{m}$-structures to $\G$-structures, $J$ is replaced by $\times$, and $\omega$ is replaced by $\ph$. Thus everything is shifted up in degree by one, so rather than consider just the single vector field $\ol{\nabla} F$, we decided to introduce a second vector field $V$. Since $\dot \alpha = \ol{\nabla} F \hk \omega = J (\ol{\nabla} F)$, and the cross product is the analogue of $J$, it was thus reasonable to try $\dot \beta = (V \times \ol{\nabla} F) \hk \ph$, which also did not work.

Finally, we considered allowing $\dot \beta$ to be quadratic in $V$ but to retain the desirable property that it be linear in $\ol{\nabla} F$. This led us quickly to $\dot \beta = V \w (V \times \ol{\nabla} F) = V \w (\ol{\nabla} F \hk V \hk \ph)$, which as we saw above does indeed work to prove Theorem~\ref{thm:assoc-B}. It was then clear what the proper generalizations~\eqref{eq:coassoc-test-variation} and~\eqref{eq:Cayley-test-variation} should be in the coassociative and Cayley cases, although there are several added complications in the Cayley case which are discussed in Section~\ref{sec:Spin7}.
\end{rmk}

\subsubsection{Coassociative submanifolds of manifolds with $\G$-structure} \label{sec:coassoc}

Let $k=4$, so that $M$ is a compact oriented embedded $4$-dimensional submanifold of $\ol{M}$. We consider the functional $\cV$ of~\eqref{eq:functional} in this setting, and consider the following restricted class of variations of the ambient metric $\ol{g}$.

Let $\gamma_t$ be a family of compactly supported smooth $3$-forms on $\ol{M}$ with $\gamma_{t=0} = 0$. Then $\ps_t = \ps + d \gamma_t$ is a one-parameter family of smooth $4$-forms on $\ol{M}$ with $\ps_{t=0} = \ps$. Note that if $d \ps = 0$, then $\ps_t$ lies in the same cohomology class as $\ps$ for all $t$, but we do not assume that $d \ps = 0$.

The space of nondegenerate $4$-forms on $\ol{M}$ is an open subbundle of the bundle $\Omega^4$ of smooth $4$-forms. Hence, for small $|t|$, the $4$-form $\ps_t$ is still nondegenerate. Consider the Riemannian metric $\ol{g}_t$ on $\ol{M}$ where
\begin{equation} \label{eq:coassoc-metric}
\text{$\ol{g}_t$ is the Riemannian metric induced by $\ps_t = \ps + d \gamma_t$}.
\end{equation}
(Note that while a nondegenerate $3$-form $\ph$ determines both a metric and an orientation, a nondegenerate $4$-form $\ps$ determines only a metric. Regardless, the choice of orientation on $\ol{M}$ can be taken to be fixed and is irrelevant to us.) Write $\dot{\gamma}$ for $\rest{\ddt}{t=0} \gamma_t$. We therefore have
\begin{equation} \label{eq:coassoc-form-time}
\rest{\ddt}{t=0} \ps_t = d \dot{\gamma}.
\end{equation}

The precise nonlinear formula for $\ol{g}_t$ in terms of $\ps_t = \ps + d \gamma_t$ is quite complicated, but just as in \S\ref{sec:assoc}, all that matters is the first variation $\rest{\ddt}{t=0} \ol{g}_t$. Using the initial $\G$-structure, the space $\Omega^4$ of smooth $4$-forms on $\ol{M}$ decomposes orthogonally with respect to $\ol{g}$ into
$$ \Omega^4 = \Omega^4_1 \oplus \Omega^4_{27} \oplus \Omega^4_7.$$
A $4$-form $\rho$ on $\ol{M}$ therefore decomposes as $\rho = \rho_{1+27} + \rho_7$, where~\cite{K-flows, K-intro} we have
\begin{equation} \label{eq:G2-4-form-decomp}
\rho_{1+27} = \frac{1}{2} h_{ij} e_i \w (e_j \hk \ps), \qquad \rho_7 = \frac{1}{2} X \w \ph,
\end{equation}
for some unique vector field $X$ and unique symmetric $2$-tensor $h$ on $\ol{M}$. (The factor of $\frac{1}{2}$ is chosen in~\eqref{eq:G2-4-form-decomp} to eliminate a factor of $2$ in~\eqref{eq:coassoc-metric-time} below.) It is easy to show, exactly as in the calculation for flows of $\Spin{7}$-structures in~\cite{K-flows-SP}, that
\begin{equation} \label{eq:coassoc-metric-time}
\rest{\ddt}{t=0} \ps_t = \rho \quad \implies \quad \rest{\ddt}{t=0} \ol{g}_t = h,
\end{equation}
where $h$ is related to $\rho$ by~\eqref{eq:G2-4-form-decomp}.

\begin{lemma} \label{lemma:h-from-rho}
Let $\rho$, $X$, and $h$ be as in~\eqref{eq:G2-4-form-decomp}. Define the $2$-tensor $\wh{\rho}$ on $\ol{M}$ by $\wh{\rho}_{pq} = \rho_{pijk} \ps_{qijk}$. Then we have
\begin{equation} \label{eq:h-from-rho}
h_{il} = \frac{1}{12} ( \wh{\rho}_{il} + \wh{\rho}_{li} ) - \frac{1}{48} (\tr_{\ol{g}} \wh{\rho}) \ol{g}_{il}.
\end{equation}
\end{lemma}
\begin{proof}
From~\eqref{eq:G2-4-form-decomp}, we have
\begin{equation} \label{eq:rho-temp}
\begin{aligned}
2 \rho_{ijkl} & = h_{ip} \ps_{pjkl} + h_{jp} \ps_{ipkl} + h_{kp} \ps_{ijpl} + h_{lp} \ps_{ijkp} \\
& \qquad {} + X_i \ph_{jkl} - X_j \ph_{ikl} + X_k \ph_{ijl} - X_l \ph_{ijk}.
\end{aligned}
\end{equation}
Using~\eqref{eq:rho-temp}, the symmetry of $h$, and the identities in~\eqref{eq:G2-identities}, we compute
\begin{align*}
2 \wh{\rho}_{im} & = (h_{ip} \ps_{pjkl} + h_{jp} \ps_{ipkl} + h_{kp} \ps_{ijpl} + h_{lp} \ps_{ijkp}) \ps_{mjkl} \\
& \qquad {} + (X_i \ph_{jkl} - X_j \ph_{ikl} + X_k \ph_{ijl} - X_l \ph_{ijk}) \ps_{mjkl} \\
& = h_{ip} (24 \ol{g}_{pm}) + 3 h_{jp} (4 \ol{g}_{im} \ol{g}_{pj} - 4 \ol{g}_{ij} \ol{g}_{pm} - 2 \ps_{ipmj}) + 0  - 3 X_j (-4 \ph_{imj}) \\
& = 12 h_{im} + 12 (\tr_{\ol{g}} h) \ol{g}_{im} + 12 X_p \ph_{pim}. 
\end{align*}
We deduce that $2 \tr_{\ol{g}} \wh{\rho} = 96 \tr_{\ol{g}} h$ and $2 \wh{\rho}_{im} + 2 \wh{\rho}_{mi} = 24 h_{im} + 24 (\tr_{\ol{g}} h) \ol{g}_{im}$. Equation~\eqref{eq:h-from-rho} follows.
\end{proof}

\begin{thm}[Theorem A for coassociative submanifolds] \label{thm:coassoc-A}
Let $M$ be a coassociative submanifold of $\ol{M}$. For small $|t|$, let $\ol{g}_t$ be a family of metrics on $\ol{M}$ induced by $\ps_t = \ps + d \gamma_t$ with $\gamma_{t=0} = 0$. Then $\rest{\ddt}{t=0} \cV(\ol{g}_t) = 0$, so $\ol{g}$ is a critical point of $\cV$ with respect to variations in this class.
\end{thm}
\begin{proof}
Since $M$ is coassociative, we can choose a local adapted frame $\{ e_1, \ldots, e_7 \}$ for $(\ol{M}, \ol{g})$ so that, along $M$, the last $4$ elements are an oriented local orthonormal frame for $(M, g)$ and the first $3$ elements are an oriented local orthonormal frame for the normal bundle of $M$ in $\ol{M}$ with respect to $\ol{g}$, and moreover that the frame is $\G$-adapted in the sense that
$$ e_3 = e_1 \times e_2, \quad e_5 = e_1 \times e_4, \quad e_6 = e_2 \times e_4, \quad e_7 = (e_1 \times e_2) \times e_4.$$
With respect to such a frame, along $M$ we have
\begin{equation} \label{eq:coassoc-A-psi}
\begin{aligned}
\ps & = e_4 \w e_5 \w e_6 \w e_7 - e_2 \w e_3 \w (e_4 \w e_5 - e_6 \w e_7) \\ & \qquad {} - e_3 \w e_1 \w (e_4 \w e_6 - e_7 \w e_5) - e_1 \w e_2 \w (e_4 \w e_7 - e_5 \w e_6).
\end{aligned}
\end{equation}
(Note that~\eqref{eq:coassoc-A-psi} is just the Hodge star of~\eqref{eq:assoc-A-phi}, since the orthogonal complement of an associative $3$-plane is a coassociative $4$-plane and vice versa.) Write $\rho$ for $\rest{\ddt}{t=0} \ps_t = d \dot{\gamma}$ and $h$ for $\rest{\ddt}{t=0} \ol{g}_t$. Then $h$ is given in terms of $\rho$ by~\eqref{eq:h-from-rho}, and we need to compute the integrand $(\tr_{g} h)  \vol_{(M, g)} = g^{ab} h_{ab}  \vol_{(M, g)}$ of~\eqref{eq:first-variation} (at $t=0$). With $\wh{\rho}$ defined as in Lemma~\ref{lemma:h-from-rho}, we compute with respect to this local frame that
\begin{equation} \label{eq:trace-hat-rho}
\tr_{\ol{g}} \wh{\rho} = \rho_{ijkl} \ps_{ijkl} = 24 (\rho_{4567} - \rho_{2345} + \rho_{2367} - \rho_{3146} + \rho_{3175} - \rho_{1247} + \rho_{1256}),
\end{equation}
where the factor of $24 = 4!$ arises because we are summing over all $1 \leq i, j, k, l \leq 7$. Similarly, with respect to this local frame, we need to compute
\begin{equation} \label{eq:hat-rho-temp}
\wh{\rho}_{aa} = \rho_{aijk} \rho_{aijk} = \rho_{4ijk} \ps_{4ijk} + \rho_{5ijk} \ps_{5ijk} + \rho_{6ijk} \ps_{6ijk} + \rho_{7ijk} \ps_{7ijk}.
\end{equation}
For example, using~\eqref{eq:coassoc-A-psi} the term $\rho_{4ijk} \ps_{4ijk}$ is $6 ( \rho_{4567} - \rho_{2345} - \rho_{3146} - \rho_{1247} )$. Combining the terms of~\eqref{eq:hat-rho-temp} and using~\eqref{eq:trace-hat-rho} we obtain
\begin{align*}
\wh{\rho}_{aa} & = 6 ( 4 \rho_{4567} - 2 \rho_{2345} + 2 \rho_{2367} - 2 \rho_{3146} + 2 \rho_{3175} - 2 \rho_{1247} + 2 \rho_{1256} ) \\
& = 12 \rho_{4567} + 12 (\rho_{4567} - \rho_{2345} + \rho_{2367} - \rho_{3146} + \rho_{3175} - \rho_{1247} + \rho_{1256}) \\
& = 12 \rho_{4567} + \frac{12}{24} \tr_{\ol{g}} \wh{\rho}.
\end{align*}
Using~\eqref{eq:h-from-rho} and the above, we compute
$$ h_{aa} = \frac{1}{6} \wh{\rho}_{aa} - \frac{4}{48} \tr_{\ol{g}} \wh \rho = \frac{1}{6} (12 \rho_{4567} + \frac{12}{24} \tr_{\ol{g}} \wh{\rho}) - \frac{2}{24} \tr_{\ol{g}} \wh \rho = 2 \rho_{4567}.$$
Hence, we have $(\tr_{g} h)  \vol_{(M, g)} = h_{aa} e_4 \w e_5 \w e_6 \w e_7 = 2 \rho_{4567} e_4 \w e_5 \w e_6 \w e_7 = 2 \rest{\rho}{M}$. Recalling that $\rho = d \dot{\gamma}$, by Stokes's Theorem and~\eqref{eq:first-variation}, we deduce that
$$ \rest{\ddt}{t=0} \cV(\ol{g}_t) = \int_M \rest{(d \dot{\gamma})}{M} = 0, $$
completing the proof.
\end{proof}

\begin{thm}[Theorem B for coassociative submanifolds] \label{thm:coassoc-B}
Suppose that $\ol{g}$ is a critical point of the functional $\cV$ with respect to variations $\ol{g}_t$ induced by $\ps_t = \ps + d \gamma_t$. Then the submanifold $M$ is coassociative.
\end{thm}
\begin{proof}
Given any compactly supported smooth $3$-form $\dot\gamma$, let $\ol{g}_t$ be the variation obtained as above by taking $\gamma_t = t \dot \gamma$ for $|t|$ small, and write $\rho$ for $d\dot\gamma = \rest{\ddt}{t=0} \ps_t $. Using~\eqref{eq:first-variation}, our hypothesis then gives that
\begin{equation} \label{eq:coassoc-B-hypothesis}
0 = \rest{\ddt}{t=0} \cV (\ol{g}_t) = \frac{1}{2} \int_M (\tr_g h) \, \vol_{(M, g)},
\end{equation}
where by~\eqref{eq:coassoc-metric-time} and Lemma~\ref{lemma:h-from-rho}, the symmetric 2-tensor $h$ is given by
\begin{equation} \label{eq:coassoc-B-start}
h_{il} = \frac{1}{12}( \wh{\rho}_{il} + \wh{\rho}_{li} ) - \frac{1}{48} (\tr_{\ol{g}} \wh{\rho}) \ol{g}_{il}, \qquad \text{where} \qquad \wh{\rho}_{il} = \rho_{imnp} \ps_{lmnp}.
\end{equation}
Taking the trace of $h$ in the tangential directions gives
\begin{align}
\tr_g h = \sum_{a=4}^7 h_{aa} & = \frac{1}{6} \sum_{a=4}^7 \wh{\rho}_{aa} - \frac{1}{12} \sum_{i=1}^7 \wh{\rho}_{ii} \nonumber \\
& = - \frac{1}{6} \sum_{a=4}^7 \rho_{amnp} \ps_{mnpa} + \frac{1}{12} \sum_{i=1}^7 \rho_{imnp} \ps_{mnpi} \nonumber \\
& = - \frac{1}{6} \rho \big( \chi(e_m, e_n, e_p)^T, e_m, e_n, e_p \big) + \frac{1}{12} \rho( \chi(e_m, e_n, e_p), e_m, e_n, e_p ) \nonumber \\
& = \frac{1}{12} \rho \big( - \chi(e_m, e_n, e_p)^T + \chi(e_m, e_n, e_p)^{\perp}, e_m, e_n, e_p \big). \label{eq:coassoc-h-trace}
\end{align}

We need to make an appropriate choice of $\dot \gamma$. Let $V, W$ be any smooth tangent vector fields on $M$ and extend them both smoothly to global vector fields with compact support, still denoted by $V, W$, on the ambient space $\ol{M}$. We choose
\begin{equation} \label{eq:coassoc-test-variation}
\dot \gamma = V \w W \w \chi(V, W, \ol{\nabla} F)
\end{equation}
where $F$ is as in Definition~\ref{defn:F}. In terms of a local frame, we have $\chi(V, W, \ol{\nabla} F)_j = V_p W_q (\ol{\nabla}_m F) \ps_{pqmj}$, and thus using the skew-symmetry of $\ps(X, Y, Z, W) = \langle \chi(X, Y, Z), W \rangle$, we have
\begin{align} \nonumber
\dot \gamma_{jkl} & = (V_j W_k - V_k W_j) \langle \chi(V, W, \ol{\nabla} F), e_l \rangle + (V_k W_l - V_l W_k) \langle \chi(V, W, \ol{\nabla} F), e_j \rangle \\
& \qquad {} + (V_l W_j - V_j W_l) \langle \chi(V, W, \ol{\nabla} F), e_k \rangle. \label{eq:gamma-coords}
\end{align}
Recall from Lemma~\ref{lemma:distance-function} that, \emph{along the submanifold $M$}, we have $\ol{\nabla} F = 0$ and $\ol{\nabla}_X (\ol{\nabla} F) = X^{\perp}$. Using these facts and equation~\eqref{eq:gamma-coords}, along $M$ we have
\begin{align}
\ol{\nabla}_i \dot \gamma_{jkl} & = (V_j W_k - V_k W_j) \langle \chi(V, W, e_i^{\perp}), e_l \rangle \nonumber \\
& \qquad {} + (V_k W_l - V_l W_k) \langle \chi(V, W, e_i^{\perp}), e_j \rangle \nonumber \\
& \qquad {} + (V_l W_j - V_j W_l) \langle \chi(V, W, e_i^{\perp}), e_k \rangle \label{eq:nabla-gamma-ijkl}
\end{align}
and similarly
\begin{align}
\ol{\nabla}_j \dot \gamma_{ikl} & = (V_i W_k - V_k W_i) \langle \chi(V, W, e_j^{\perp}), e_l \rangle \nonumber \\
& \qquad {} + (V_k W_l - V_l W_k) \langle \chi(V, W, e_j^{\perp}), e_i \rangle \nonumber \\
& \qquad {} + (V_l W_i - V_i W_l) \langle \chi(V, W, e_j^{\perp}), e_k \rangle. \label{eq:nabla-gamma-jikl}
\end{align}
Recall that
\begin{equation} \label{eq:d-gamma}
(d \dot \gamma)_{ijkl} = \ol{\nabla}_i \dot \gamma_{jkl} - \ol{\nabla}_j \dot \gamma_{ikl} + \ol{\nabla}_k \dot \gamma_{ijl} - \ol{\nabla}_l \dot \gamma_{ijk}.
\end{equation}

We need to compute the two terms on the right hand side of~\eqref{eq:coassoc-h-trace} for $\rho = d \dot \gamma$. Recall that $V = V_k e_k$ and $W = W_l e_l$ are both tangent along $M$, so $\langle V, X^T \rangle = \langle V, X \rangle$ and $\langle V, X^{\perp} \rangle = 0$ for any vector field $X$, and similarly for $W$.

First we compute $(d \dot \gamma) \big( \chi(e_j, e_k, e_l)^T, e_j, e_k, e_l)$. Because we are summing over $j, k, l$ and since $\dot \gamma$ and $\chi$ are both totally skew, from~\eqref{eq:d-gamma} we have 
\begin{align} \nonumber
(d \dot \gamma) \big( \chi(e_j, e_k, e_l)^T, e_j, e_k, e_l \big) & \phantom{:} = (\ol{\nabla} \dot \gamma) \big( \chi(e_j, e_k, e_l)^T, e_j, e_k, e_l \big) - 3 (\ol{\nabla} \dot \gamma) \big( e_j, \chi(e_j, e_k, e_l)^T, e_k, e_l \big) \\
&  =: (I) - 3 \cdot (II). \label{eq:dgamma-temp}
\end{align}
From~\eqref{eq:nabla-gamma-ijkl} with $e_i = \chi(e_j, e_k, e_l)^T$ we see that $(I) = 0$. As for $(II)$, we use~\eqref{eq:nabla-gamma-jikl} with $e_i = \chi(e_j, e_k, e_l)^T$ and the skew-symmetry of $\ps(X, Y, Z, U) = \langle \chi(X, Y, Z), U \rangle$ to compute
\begin{align*}
(II) & = - \big( \langle V, \chi(e_j, e_k, e_l) \rangle W_k - V_k \langle W, \chi(e_j, e_k, e_l) \rangle \big) \langle \chi(V, W, e_l), e_j^{\perp} \rangle \\
& \qquad {} + (V_k W_l - V_l W_k) \langle \chi(V, W, e_j^{\perp}), \chi(e_j, e_k, e_l)^T \rangle \\
& \qquad {} - (V_l \langle W, \chi(e_j, e_k, e_l) \rangle -  \langle V, \chi(e_j, e_k, e_l) \rangle W_l) \langle \chi(V, W, e_k), e_j^{\perp} \rangle \\
& = - \big( \langle \chi(e_j, W, e_l), V \rangle  - \langle \chi(e_j, V, e_l), W \rangle \big) \langle \chi(V, W, e_l), e_j^{\perp} \rangle \\
& \qquad {} + 2 \langle \chi(V, W, e_j^{\perp}), \chi(e_j, V, W)^T \rangle \\
& \qquad {} - \big( \langle \chi(e_j, e_k, V), W \rangle - \langle \chi(e_j, e_k, W), V \rangle \big) \langle \chi(V, W, e_k), e_j^{\perp} \rangle \\
& = 2 \langle \chi(W, e_l, V), e_j \rangle \langle \chi(V, W, e_l), e_j^{\perp} \rangle + 2 \langle \chi(V, W, e_j^{\perp}), \chi(V, W, e_j)^T \rangle \\
& \qquad {} + 2 \langle \chi(e_k, V, W), e_j \rangle \langle \chi(V, W, e_k), e_j^{\perp} \rangle \\
& = 4 \langle \chi(V, W, e_l)^{\perp}, \chi(V, W, e_l)^{\perp} \rangle + 2 \langle \chi(V, W, e_j^{\perp}), \chi(V, W, e_j)^T \rangle.
\end{align*}
Therefore, substituting the above into~\eqref{eq:dgamma-temp} gives
\begin{align}
& \qquad (d \dot \gamma) \big( \chi(e_j, e_k, e_l)^T, e_j, e_k, e_l \big) \nonumber \\
& = -12 \langle \chi(V, W, e_l)^{\perp}, \chi(V, W, e_l)^{\perp} \rangle - 6 \langle \chi(V, W, e_l^{\perp}), \chi(V, W, e_l)^T \rangle. \label{eq:d-gamma-ijkl-h}
\end{align}

Next we compute $(d \dot \gamma) \big( \chi(e_j, e_k, e_l)^{\perp}, e_j, e_k, e_l \big)$. As before we have 
\begin{align} \nonumber
(d \dot \gamma) \big( \chi(e_j, e_k, e_l)^{\perp}, e_j, e_k, e_l \big) & \phantom{:} = (\ol{\nabla} \dot \gamma) \big( \chi(e_j, e_k, e_l)^{\perp}, e_j, e_k, e_l \big) - 3 (\ol{\nabla} \dot \gamma) \big( e_j, \chi(e_j, e_k, e_l)^{\perp}, e_k, e_l \big) \\
& : = (I)' - 3 \cdot (II)'. \label{eq:dgamma-temp2}
\end{align}
For $(I)'$, we have by~\eqref{eq:nabla-gamma-ijkl} with $e_i = \chi(e_j, e_k, e_l)^{\perp}$ that
\begin{align*}
(I)' & = - (V_j W_k - V_k W_j) \langle \chi(V, W, e_l), \chi(e_j, e_k, e_l)^{\perp} \rangle \\
& \qquad {} - (V_k W_l - V_l W_k) \langle \chi(V, W, e_j), \chi(e_j, e_k, e_l)^{\perp} \rangle \\
& \qquad {} - (V_l W_j - V_j W_l) \langle \chi(V, W, e_k), \chi(e_j, e_k, e_l)^{\perp} \rangle \\
& = - 6 \langle \chi(V, W, e_l)^{\perp}, \chi(V, W, e_l)^{\perp} \rangle.
\end{align*}
For $(II)'$, note that in~\eqref{eq:nabla-gamma-jikl}, with $e_i = \chi(e_j, e_k, e_l)^{\perp}$, the first and third lines on the right-hand side vanish, since $V$ and $W$ are tangent to $M$. Thus we have
\begin{align*}
(II)' & = (V_k W_l - V_l W_k) \langle \chi(V, W, e_j^{\perp}), \chi(e_j, e_k, e_l)^{\perp} \rangle \\
& = 2 \langle \chi(V, W, e_j^{\perp}), \chi(V, W, e_j)^{\perp} \rangle.
\end{align*}
Therefore, substituting the above into~\eqref{eq:dgamma-temp2} gives
\begin{align}
& \qquad (d \dot \gamma) \big( \chi(e_j, e_k, e_l)^{\perp}, e_j, e_k, e_l \big) \nonumber \\
& = - 6 \langle \chi(V, W, e_l)^{\perp}, \chi(V, W, e_l)^{\perp} \rangle - 6 \langle \chi(V, W, e_l^{\perp}), \chi(V, W, e_l)^{\perp} \rangle. \label{eq:d-gamma-jikl-h}
\end{align}

Now we substitute~\eqref{eq:d-gamma-ijkl-h} and~\eqref{eq:d-gamma-jikl-h} into the right hand side of~\eqref{eq:coassoc-h-trace} for $\rho = d \dot \gamma$, to obtain 
\begin{align*}
\tr_g h & = \frac{1}{2} \langle \chi(V, W, e_l)^{\perp}, \chi(V, W, e_l)^{\perp} \rangle + \frac{1}{2} \langle \chi(V, W, e_l^{\perp}), \chi(V, W, e_l)^T \rangle \\
& \qquad {} - \frac{1}{2} \langle \chi(V, W, e_l^{\perp}), \chi(V, W, e_l)^{\perp} \rangle.
\end{align*}
Because $\{ e_1, e_2, e_3 \}$ are normal and $\{ e_4, e_5, e_6, e_7 \}$ are tangent, the above becomes
\begin{align*}
\tr_g h & = \frac{1}{2} \sum_{k=1}^3 | \chi(V, W, e_k)^{\perp} |^2 + \frac{1}{2} \sum_{a=4}^7 | \chi(V, W, e_a)^{\perp} |^2 \\
& \qquad {} + \frac{1}{2} \sum_{l=1}^3 \langle \chi(V, W, e_l), \chi(V, W, e_l)^T \rangle - \frac{1}{2} \sum_{l=1}^3 | \chi(V, W, e_l)^{\perp} |^2 \\
& = \frac{1}{2} \sum_{a=4}^7 | \chi(V, W, e_a)^{\perp} |^2 + \frac{1}{2} \sum_{l=1}^3 | \chi(V, W, e_l)^T |^2.
\end{align*}
Noting that
$$ \sum_{l=1}^3 | \chi(V, W, e_l)^T |^2 = \sum_{l=1}^3 \sum_{a=4}^7 \langle \chi(V, W, e_l), e_a \rangle^2 = \sum_{a=4}^7 \sum_{l=1}^3 \langle \chi(V, W, e_a), e_l \rangle^2 = \sum_{a=4}^7| \chi(V, W, e_a)^{\perp} |^2, $$
we deduce that
$$ \tr_g h = \sum_{a=4}^7 | \chi(V, W, e_a)^{\perp} |^2. $$
Hence from the above and~\eqref{eq:coassoc-B-hypothesis}, we finally obtain
\begin{equation*}
0 = \rest{\ddt}{t=0} \cV(\ol{g}_t) = \frac{1}{2} \int_M \left( \sum_{a=4}^7 | \chi(V, W, e_a)^{\perp} |^2 \right) \vol_{(M, g)}.
\end{equation*}
The integrand must vanish as it is nonnegative. Hence $\chi(V, W, e_a)^{\perp} = 0$ for $a = 4, 5, 6, 7$ and \emph{any tangent vector fields $V, W$ on $M$}. Taking $V, W$ to be $e_4, e_5, e_6, e_7$ shows that $\chi(e_a, e_b, e_c)$ is tangent to $M$ for $4 \leq a, b, c \leq 7$ and hence by Proposition~\ref{prop:coassoc-condition} we conclude that $M$ is coassociative.
\end{proof}

\subsection{Cayley submanifolds of manifolds with $\Spin{7}$-structure} \label{sec:Spin7}

In this section we consider the case when $\ol{M}$ has a $\Spin{7}$-structure. Suppose that $n=8$ and let $(\Ph, \ol{g})$ be a $\Spin{7}$-structure on $\ol{M}$. Here $\Ph$ is a nondegenerate $4$-form which determines the Riemannian metric $\ol{g}$ in a nonlinear way. It also induces an orientation with respect to which $\Ph$ is self-dual, meaning $\sta \Ph = \Ph$. See~\cites{K1, K-flows-SP} for more about $\Spin{7}$-structures. We do not assume anything about the \emph{torsion} $\ol{\nabla} \Ph$ of the $\Spin{7}$-structure.

There is a $3$-fold vector cross product $P$ on $\ol{M}$ defined by
$$ \Ph(X, Y, Z, W) = \ol{g}( P(X, Y, Z), W ).$$
That is, $P$ is obtained from $\Ph$ by raising an index using $\ol{g}$. It is clear that $P(X, Y, Z)$ is orthogonal to each of $X, Y, Z$. Moreover, $| P(X,Y,Z) |^2 = | X \w Y \w Z|^2$.

Recall from Definition~\ref{defn:structures} that a $4$-dimensional orientable submanifold $M$ of $(\ol{M}, \ol{g}, \Ph)$ is \emph{Cayley} if its tangent spaces $T_x M$ are preserved by the cross product --- that is, if $P(X, Y, Z)$ is tangent to $M$ whenever $X, Y, Z$ are all tangent to $M$. This is equivalent to the existence of an orientation on $M$ with respect to which $\rest{\Ph}{M} = \vol_M$. To see this, let $X, Y, Z, W$ be orthonormal tangent vectors to $M$ and observe by Cauchy--Schwarz that
$$ | \Ph(X, Y, Z, W) | = | \ol{g} (P(X, Y, Z), W) | \leq | P(X, Y, Z) | \, |W| = | X \w Y \w Z | \, | W | = 1 = | \vol_M (X, Y, Z, W) |, $$
with equality if and only if $P(X, Y, Z) = \pm W$. Note that when $d \Ph = 0$, a Cayley submanifold is \emph{calibrated} by $\Ph$ and hence minimal~\cite{HL}, but we do not assume $d \Ph = 0$.

\textbf{Note.} In this section, our index ranges are $1 \leq i, j, k, l, m, p, q \leq 8$. Moreover, $\{ e_1, \ldots, e_8 \}$ is a local orthonormal frame for $\ol{M}$, so $\ol{g}_{ij} = \delta_{ij}$. Thus all our indices will be subscripts and we sum over repeated indices. We have $1 \leq a, b \leq 4$ as the ranges for the indices tangent to the submanifold $M$. We explicitly display the summation symbols and ranges in equations that feature sums over different ranges, in the proof of Theorem~\ref{thm:Cayley-B}.

We need to make use of various contraction identities for $\Spin{7}$-structures, which can be found in~\cite{K-flows-SP}. The identities we need are:
\begin{equation} \label{eq:SP-identities}
\begin{aligned}
\Ph_{ijpq} \Ph_{klpq} & = 6 \ol{g}_{ik} \ol{g}_{jl} - 6 \ol{g}_{il} \ol{g}_{jk} - 4 \Ph_{ijkl}, \\
\Ph_{impq} \Ph_{jmpq} & = 42 \ol{g}_{ij}.
\end{aligned}
\end{equation}

Let $k=4$, so that $M$ is a compact oriented embedded $4$-dimensional submanifold of $\ol{M}$. We consider the functional $\cV$ of~\eqref{eq:functional} in this setting. We will show that in this case, the correct restricted class of variations of the ambient metric $\ol{g}$ that makes Theorems A and B true is \emph{somewhat different}, in two ways, from the other three cases.

One reason for the increased complexity in the $\Spin{7}$ case is the following. In contrast to the $\G$ case, the space $\Omega^4_+$ of nondegenerate $4$-forms on a manifold that admits $\Spin{7}$-structures is \emph{not} open, but rather it is nonlinearly embedded in the space $\Omega^4$ of smooth $4$-forms on $\ol{M}$. In fact, using the initial $\Spin{7}$-structure $\Ph$, the space $\Omega^4$ of smooth $4$-forms on $\ol{M}$ decomposes orthogonally~\cite{K-flows-SP} with respect to $\ol{g}$ into
$$ \Omega^4 = \Omega^4_1 \oplus \Omega^4_{35} \oplus \Omega^4_7 \oplus \Omega^4_{27}.$$
Here $\Omega^4_{35}$ is the space of \emph{anti-self-dual} $4$-forms, and $\Omega^4_{1+7+27}$ is the space of \emph{self-dual} $4$-forms. It is shown in~\cite{K-flows-SP} that $\Omega^4_{1+35+7}$ is precisely the tangent space $T_{\Ph} \Omega^4_+$ at $\Ph$ to the space $\Omega^4_+$ of nondegenerate $4$-forms. That is, $\sigma \in \Omega^4_{1+35+7}$ if and only if there exists a one-parameter family of $\Spin{7}$-structures $\Ph_t$ on $\ol{M}$ with $\Ph_0 = \Ph$ and
$$ \sigma = \rest{\ddt}{t=0} \Ph_t. $$

\begin{rmk} \label{rmk:Um-not-open}
In the $\U{m}$ case, it is also true that the space $\Omega^{(1,1)}_+$ of positive $(1,1)$-forms is not open in $\Omega^2$, but it \emph{is open} in $\Omega^{(1,1)}$. Since we \emph{did not vary} $J$, and took $\omega_t = \wt \omega_t^{(1,1)}$, the $\U{m}$ case is actually similar to the $\G$ cases.
\end{rmk}

The $\Omega^4_1$ directions in $T_{\Ph} \Omega^4_+$ are precisely the directions which scale the metric, while the $\Omega^4_{35}$ directions correspond to metric variations which infinitesimally preserve the volume, and the $\Omega^4_7$ directions correspond to variations of the $\Spin{7}$-structure which infinitesimally do not change the metric at all.

We show in Theorems~\ref{thm:Cayley-A} and~\ref{thm:Cayley-B} below that in order to prove versions of Theorems A and B in this case, we must restrict to variations $\Ph_t$ (with $\Ph_0 = \Ph$) of the special form
$$ \rest{\ddt}{t=0} \Ph_t = \pi_{35+7} d \dot{\gamma} $$
for some $3$-form $\dot{\gamma}$ on $\ol{M}$. That is, these must be variations in the direction of the orthogonal projection of an \emph{exact} $4$-form onto the \emph{infinitesimally volume preserving directions}. (See also Remark~\ref{rmk:Cayley-differences}.)

Note that it \emph{may not} be possible to choose $\Ph_t$ such that $\Ph_t - \Ph$ is exact, as we were able to do in both the $\U{m}$ and $\G$ cases. Regardless, since we do not assume that $d \Ph = 0$, there is not necessarily any well-defined cohomology class to preserve anyway.

Consider the Riemannian metric $\ol{g}_t$ on $\ol{M}$ where
\begin{equation} \label{eq:SP-metric}
\text{$\ol{g}_t$ is the Riemannian metric induced by $\Ph_t$ where $\rest{\ddt}{t=0} \Ph_t = \pi_{35+7} d \dot{\gamma}$}.
\end{equation}

As mentioned above, a general $4$-form $\sigma$ on $\ol{M}$ decomposes as $\sigma = \sigma_{1+35} + \sigma_7 + \sigma_{27}$, where~\cite{K-flows-SP} we have
\begin{equation} \label{eq:SP-4-form-decomp}
\sigma_{1+35} = \frac{1}{2} h_{ij} e_i \w (e_j \hk \Ph), \qquad \sigma_7 = \frac{1}{2} \beta_{ij} e_i \w (e_j \hk \Ph),
\end{equation}
for some unique symmetric $2$-tensor $h$ on $\ol{M}$ and some unique $2$-form $\beta$ of type $\Omega^2_7$. (The factor of $\frac{1}{2}$ is chosen in~\eqref{eq:SP-4-form-decomp} to eliminate a factor of $2$ in~\eqref{eq:Cayley-metric-time} below.) Moreover, it is also shown in~\cite{K-flows-SP} that
\begin{equation} \label{eq:SP-27}
\sigma \in \Omega^4_{27} \, \iff \, \sigma_{ijkl} \Ph_{mjkl} = 0.
\end{equation}
Note that we can decompose $h = \frac{1}{8} (\tr_{\ol{g}} h) \ol{g} + h^0$ into its pure-trace and trace-free parts, which are orthogonal with respect to $\ol{g}$, and which correspond to the $\Omega^4_1$ and $\Omega^4_{35}$ components of $\sigma$, respectively. But for the consideration of the first variation of the metric, it is natural to combine these two components as we have in~\eqref{eq:SP-4-form-decomp} above. It is shown in~\cite{K-flows-SP} that
\begin{equation} \label{eq:Cayley-metric-time}
\rest{\ddt}{t=0} \Ph_t = \sigma \in \Omega^4_{1+35+7} \quad \implies \quad \rest{\ddt}{t=0} \ol{g}_t = h,
\end{equation}
where $h$ is related to $\sigma$ by~\eqref{eq:SP-4-form-decomp}.

\begin{lemma} \label{lemma:h-from-sigma}
Let $\sigma \in \Omega^4$ and $h$ be as in~\eqref{eq:SP-4-form-decomp}. Define the $2$-tensor $\wh{\sigma}$ on $\ol{M}$ by $\wh{\sigma}_{pq} = \sigma_{pijk} \Ph_{qijk}$. Then we have
\begin{equation} \label{eq:h-from-sigma}
h_{im} = \frac{1}{24} ( \wh{\sigma}_{im} + \wh{\sigma}_{mi} ) - \frac{1}{112} (\tr_{\ol{g}} \wh{\sigma}) \ol{g}_{im}.
\end{equation}
Moreover, the trace-free part $h^0$ of $h$, which gives $\sigma_{35} = \frac{1}{2} h^0_{ij} e_i \w (e_j \hk \Ph)$, is given by
\begin{equation} \label{eq:h0-from-sigma}
h^0_{im} = \frac{1}{24} ( \wh{\sigma}_{im} + \wh{\sigma}_{mi} ) - \frac{1}{96} (\tr_{\ol{g}} \wh{\sigma}) \ol{g}_{im}.
\end{equation}
\end{lemma}
\begin{proof}
From~\eqref{eq:SP-4-form-decomp}, we have
\begin{equation} \label{eq:sigma-temp}
2 \sigma_{ijkl} = A_{ip} \Ph_{pjkl} + A_{jp} \Ph_{ipkl} + A_{kp} \Ph_{ijpl} + A_{lp} \Ph_{ijkp} + 2 (\sigma_{27})_{ijkl}
\end{equation}
where $A_{pq} = h_{pq} + \beta_{pq}$ in terms of symmetric and skew-symmetric parts, respectively. From~\eqref{eq:SP-27} we know that $(\sigma_{27})_{ijkl} \Ph_{mjkl} = 0$. Thus, using~\eqref{eq:sigma-temp}, and the identities in~\eqref{eq:SP-identities}, we compute
\begin{align*}
2 \wh{\sigma}_{im} & = (A_{ip} \Ph_{pjkl} + A_{jp} \Ph_{ipkl} + A_{kp} \Ph_{ijpl} + A_{lp} \Ph_{ijkp}) \Ph_{mjkl} + 0 \\
& = A_{ip} (42 \ol{g}_{pm}) + 3 A_{jp} (6 \ol{g}_{im} \ol{g}_{pj} - 6 \ol{g}_{ij} \ol{g}_{pm} - 4 \Ph_{ipmj}) \\
& = 24 A_{im} + 18 (\tr_{\ol{g}} A) \ol{g}_{im} - 12 A_{jp} \Ph_{jpim}.
\end{align*}
We deduce that $2 \tr_{\ol{g}} \wh{\sigma} = 168 \tr_{\ol{g}} h$ and $2 \wh{\sigma}_{im} + 2 \wh{\sigma}_{mi} = 48 h_{im} + 36 (\tr_{\ol{g}} h) \ol{g}_{im}$. Equation~\eqref{eq:h-from-sigma} follows.

Then we compute
\begin{align*}
h^0_{im} & = h_{im} - \frac{1}{8} (\tr_{\ol{g}} h) \ol{g}_{im} = h_{im} - \frac{1}{8 \cdot 84} (\tr_{\ol{g}} \wh{\sigma}) \ol{g}_{im} \\
& = \frac{1}{24} ( \wh{\sigma}_{im} + \wh{\sigma}_{mi} ) - \Big( \frac{1}{112} + \frac{1}{8 \cdot 84} \Big) (\tr_{\ol{g}} \wh{\sigma}) \ol{g}_{im}
\end{align*}
which simplifies to~\eqref{eq:h0-from-sigma}.
\end{proof}

\begin{thm}[Theorem A for Cayley submanifolds] \label{thm:Cayley-A}
Let $M$ be a Cayley submanifold of $\ol{M}$. For small $|t|$, let $\ol{g}_t$ be a family of metrics on $\ol{M}$ induced by $\Ph_t$ where $\rest{\ddt}{t=0} \Ph_t = \pi_{35+7} d \dot{\gamma}$ and $\Ph_{0} = \Ph$. Suppose also that the following \emph{additional condition} holds:
\begin{equation} \label{eq:Cayley-condition}
\int_M \rest{(\sta d \dot \gamma)}{M} = 0. \quad \text{(Note that this condition \emph{depends} on the submanifold $M$.)}
\end{equation}
Then $\rest{\ddt}{t=0} \cV(\ol{g}_t) = 0$, so $\ol{g}$ is a critical point of $\cV$ with respect to variations in this class.
\end{thm}
\begin{proof}
Since $M$ is Cayley, we can choose a local adapted frame $\{ e_1, \ldots, e_8 \}$ for $(\ol{M}, \ol{g})$ so that, along $M$, the first $4$ elements are an oriented local orthonormal frame for $(M, g)$ and the last $4$ elements are an oriented local orthonormal frame for the normal bundle of $M$ in $\ol{M}$ with respect to $\ol{g}$, and moreover that the frame is $\Spin{7}$-adapted. With respect to such a frame, along $M$ we have
\begin{equation} \label{eq:Cayley-A-Phi}
\begin{aligned}
\Ph & = e_1 \w e_2 \w e_3 \w e_4 + (e_1 \w e_2 - e_3 \w e_4) \w (e_5 \w e_6 - e_7 \w e_8) \\
& \qquad {} + (e_1 \w e_3 - e_4 \w e_2) \w (e_5 \w e_7 - e_8 \w e_6) \\
& \qquad {} + (e_1 \w e_4 - e_2 \w e_3) \w (e_5 \w e_8 - e_6 \w e_7) + e_5 \w e_6 \w e_7 \w e_8.
\end{aligned}
\end{equation}
We need to compute the integrand $(\tr_{g} h)  \vol_{(M, g)} = g^{ab} h_{ab}  \vol_{(M, g)}$ of~\eqref{eq:first-variation} (at $t=0$), with $h = h(\sigma)$ depending linearly on $\sigma$ via~\eqref{eq:h-from-sigma}. For now we take $\sigma$ to be any $4$-form in $\Omega^4_{1 + 35 + 7}$, and we will later specialize to $\sigma = \rest{\ddt}{t=0} \Ph_t = \pi_{7 + 35} d \dot{\gamma}$. With $\wh{\sigma}$ defined as in Lemma~\ref{lemma:h-from-sigma}, we compute with respect to this local frame that
\begin{equation} \label{eq:trace-hat-sigma}
\tr_{\ol{g}} \wh{\sigma} = \sigma_{ijkl} \Ph_{ijkl} = 24 (\sigma_{1234} + A(\sigma) + \sigma_{5678})
\end{equation}
where
\begin{equation} \label{eq:sigma-A-temp}
\begin{aligned}
A(\sigma) & = \sigma_{1256} - \sigma_{1278} - \sigma_{3456} + \sigma_{3478} + \sigma_{1357} - \sigma_{1386} - \sigma_{4257} + \sigma_{4286} \\
& \qquad {} + \sigma_{1458} - \sigma_{1467} - \sigma_{2358} + \sigma_{2367}
\end{aligned}
\end{equation}
consists of the 12 terms involving two tangent and two normal indices. In~\eqref{eq:trace-hat-sigma}, the factor of $24 = 4!$ arises because we are summing over all $1 \leq i, j, k, l \leq 8$. Similarly, with respect to this local frame, we need to compute
\begin{equation} \label{eq:hat-sigma-temp}
\wh{\sigma}_{aa} = \sigma_{aijk} \Ph_{aijk} = \sigma_{1ijk} \Ph_{1ijk} + \sigma_{2ijk} \Ph_{2ijk} + \sigma_{3ijk} \Ph_{3ijk} + \sigma_{4ijk} \Ph_{4ijk}.
\end{equation}
We can compute this as follows. First, since we are summing over all $1 \leq i, j, k \leq 8$, we get an overall factor of $6 = 3!$. Looking at~\eqref{eq:Cayley-A-Phi}, the term $e_1 \w e_2 \w e_3 \w e_4$ will contribute to every one of the four terms in~\eqref{eq:hat-sigma-temp}, since all four indices $1,2,3,4$ appear. The term $e_5 \w e_6 \w e_7 \w e_8$ will not contribute at all. Finally, the middle twelve terms, which have two tangent and two normal indices, will each contribute to two of the four terms in~\eqref{eq:hat-sigma-temp}. In summary, we have
\begin{equation} \label{eq:hat-sigma-temp2}
\wh{\sigma}_{aa} = 6 ( 4 \sigma_{1234} + 2 A(\sigma) ),
\end{equation}
where $A(\sigma)$ is given by~\eqref{eq:sigma-A-temp}.

Using~\eqref{eq:h-from-sigma} and equations~\eqref{eq:trace-hat-sigma} and~\eqref{eq:hat-sigma-temp2}, we compute
\begin{align} \nonumber
h(\sigma)_{aa} & = \frac{1}{12} \wh{\sigma}_{aa} - \frac{4}{112} \tr_{\ol{g}} \wh \sigma \\ \nonumber
& = \frac{1}{12} 6 ( 4 \sigma_{1234} + 2 A(\sigma) ) - \frac{1}{28} 24 (\sigma_{1234} + A(\sigma) + \sigma_{5678}) \\
& = \frac{8}{7} \sigma_{1234} + \frac{1}{7} A(\sigma) - \frac{6}{7} \sigma_{5678}. \label{eq:Cayley-trgh}
\end{align}
Since by~\eqref{eq:h-from-sigma} only the symmetric part of $\wh{\sigma}$ contributes to $h$, and by~\eqref{eq:SP-4-form-decomp} the skew part of $\wh{\sigma}$ comes from $\sigma_7$, we deduce that $h(\cdot)$ annihilates $\Omega^4_{7}$. Hence, recalling that $\sigma \in \Omega^4_{1 + 35 + 7}$, we have
\begin{equation} \label{eq:CaleyA-1}
h(\sigma)_{aa} = h(\sigma_{1})_{aa} + h(\sigma_{35})_{aa}.
\end{equation}
From the fact that $\Omega^4_1 = \{ f \Ph \mid f \in C^{\infty}(\ol{M}) \}$, we deduce that $A(\sigma_1) = 12(\sigma_1)_{1234}$ and that $(\sigma_1)_{1234} = (\sigma_1)_{5678}$, whence by~\eqref{eq:Cayley-trgh} with $\sigma_1$ in place of $\sigma$, we have
\begin{equation} \label{eq:CaleyA-2}
h(\sigma_1)_{aa} = 2(\sigma_1)_{1234}.
\end{equation}
On the other hand, because $\sigma_{35}$ is \emph{anti-self-dual}, we have $(\sigma_{35})_{5678} = - (\sigma_{35})_{1234}$, and it is easy to see that the expression~\eqref{eq:sigma-A-temp} for $A(\sigma_{35})$ vanishes, because the 12 terms cancel in pairs. [For example, $(\sigma_{35})_{3478} = - (\sigma_{35})_{1256}$.] Thus we have
\begin{equation} \label{eq:CaleyA-3}
h(\sigma_{35})_{aa} = \frac{8}{7} (\sigma_{35})_{1234} + 0 - \frac{6}{7} ( - (\sigma_{35})_{1234}) = 2 (\sigma_{35})_{1234}.
\end{equation}
Finally, it is easy to see from the second equation in~\eqref{eq:SP-4-form-decomp} and~\eqref{eq:Cayley-A-Phi} that
\begin{equation} \label{eq:CayleyA-4}
\rest{\sigma_7}{M} = 0, \quad \text{so} \quad (\sigma_7)_{1234} = 0.
\end{equation}
Using~\eqref{eq:CaleyA-1}--\eqref{eq:CayleyA-4}, equation~\eqref{eq:Cayley-trgh} becomes
\begin{equation}\label{eq:Cayley-trgh-processed}
h(\sigma)_{aa} = 2 (\sigma_{1 + 35})_{1234} = 2\sigma_{1234}.
\end{equation}

To continue we specialize to $\sigma = \rest{\ddt}{t=0} \Ph_t = \pi_{7 + 35} d \dot{\gamma}$. This is where we need to use our hypothesis that the velocity $\pi_{7 + 35} d \dot \gamma$ has no $\Omega^4_1$ component in this case, to be able to ensure that such first variations vanish when $M$ is Cayley. Specifically, by the first equality in~\eqref{eq:Cayley-trgh-processed} we have that
$$(\tr_{g} h)  \vol_{(M, g)} = h_{aa} e_1 \w e_2 \w e_3 \w e_4 = 2 (\pi_{35} d \dot \gamma)_{1234} e_1 \w e_2 \w e_3 \w e_4 = 2 (\rest{\pi_{35} d \dot \gamma)}{M}.$$
Since
$$ 2 \pi_{35} d \dot \gamma = 2 \pi_- d \dot \gamma = (d \dot \gamma - \sta d \dot \gamma), $$
from~\eqref{eq:first-variation} and Stokes's Theorem we deduce that
$$ \rest{\ddt}{t=0} \cV(\ol{g}_t) = \frac{1}{2} \int_M \rest{(d \dot \gamma - \sta d \dot \gamma)}{M} = 0 - \frac{1}{2} \int_M \rest{(\sta d \dot \gamma)}{M}. $$
Thus, further imposing the additional condition from~\eqref{eq:Cayley-condition} completes the proof.
\end{proof}
\begin{rmk} \label{rmk:Cayley-velocity-not-exact}
Here we encounter a \emph{marked difference} from the associative and coassociative cases. Even though, by~\eqref{eq:Cayley-trgh-processed}, $\frac{1}{2}(\tr_g h) \vol_{(M, g)}$ is still precisely the restriction to $M$ of the velocity $\dot \mu_{t=0}$ of the path of calibration forms $\mu_t$, whether it can be made exact on $\ol{M}$ is not clear and application of Stokes's Theorem does not seem possible in general without the additional assumption~\eqref{eq:Cayley-condition}. See also Remark~\ref{rmk:Cayley-cannot-fix}.
\end{rmk}

\begin{thm}[Theorem B for Cayley submanifolds] \label{thm:Cayley-B}
Suppose that $\ol{g}$ is a critical point of the functional $\cV$ with respect to variations $\ol{g}_t$ induced by $\Ph_t$ where $\rest{\ddt}{t=0} \Ph_t = \pi_{35+7} d \dot{\gamma}$, with $\dot \gamma$ satisfying the additional  condition~\eqref{eq:Cayley-condition}. Then the submanifold $M$ is Cayley.
\end{thm}
\begin{proof}
Given any smooth $3$-form $\dot\gamma$ such that $\int_M \rest{(\sta d \dot \gamma)}{M} = 0$, as noted right before Remark~\ref{rmk:Um-not-open} there exists a family $\Ph_t$ as above inducing a variation $\ol{g}_t$. Writing $\sigma$ for $\pi_{35+7} d \dot \gamma = \rest{\ddt}{t=0} \Ph_t$ and using~\eqref{eq:first-variation}, our hypothesis then gives that
\begin{equation} \label{eq:Cayley-B-hypothesis}
\rest{\ddt}{t=0} \cV (\ol{g}_t) = \frac{1}{2} \int_M (\tr_g h) \, \vol_{(M, g)} = 0,
\end{equation}
with $h = h(\sigma)$ given in terms of $\sigma$ by~\eqref{eq:h-from-sigma}. We observe that since $\sigma_1 = 0$, we have
\begin{equation}\label{eq:h0-suffices}
h(\sigma) = h(\sigma - \sigma_1) = h(\sigma)^0,
\end{equation}
where by~\eqref{eq:Cayley-metric-time} and~\eqref{eq:h0-from-sigma}, the traceless symmetric tensor $h(\sigma)^0$ is given by
\begin{equation} \label{eq:Cayley-B-start}
h(\sigma)^0_{il} = \frac{1}{24}( \wh{\sigma}_{il} + \wh{\sigma}_{li} ) - \frac{1}{96} (\tr_{\ol{g}} \wh{\sigma}) \ol{g}_{il}, \qquad \text{where} \qquad \wh{\sigma}_{il} = \sigma_{imnp} \Ph_{lmnp}.
\end{equation}
Taking the trace of $h(\sigma)^0$ in the tangential directions gives
\begin{align}
\tr_g h(\sigma)^0 = \sum_{a=1}^4 h(\sigma)^0_{aa} & = \frac{1}{12} \sum_{a=1}^4 \wh{\sigma}_{aa} - \frac{1}{24} \sum_{i=1}^8 \wh{\sigma}_{ii} \nonumber \\
& = - \frac{1}{12} \sum_{a=1}^4 \sigma_{amnp} \Ph_{mnpa} + \frac{1}{24} \sum_{i=1}^8 \sigma_{imnp} \Ph_{mnpi} \nonumber \\
& = - \frac{1}{12} \sigma \big( P(e_m, e_n, e_p)^T, e_m, e_n, e_p \big) + \frac{1}{24} \sigma( P(e_m, e_n, e_p), e_m, e_n, e_p ) \nonumber \\
& = \frac{1}{24} \sigma \big( - P (e_m, e_n, e_p)^T + P (e_m, e_n, e_p)^{\perp}, e_m, e_n, e_p \big). \label{eq:Cayley-h0-trace}
\end{align}
We need to make an appropriate choice of $\dot \gamma$. We proceed as in the coassociative case of Theorem~\ref{thm:coassoc-B}, replacing $\chi$ with $P$. Explicitly, let $V, W$ be any smooth tangent vector fields on $M$ and extend them both smoothly to global vector fields with compact support, still denoted by $V, W$, on the ambient space $\ol{M}$. We choose
\begin{equation} \label{eq:Cayley-test-variation}
\dot \gamma = V \w W \w P(V, W, \ol{\nabla} F)
\end{equation}
where $F$ is as in Definition~\ref{defn:F}. Note that while it is true that $d \dot \gamma$ is in general not of type $\Omega^4_{35+7}$, its $\Omega^4_{27}$ component drops out in the computation of $\wh{\sigma}$, and its $\Omega^4_1$ component drops out in the computation of $h^0$. Consequently,
\begin{equation} \label{eq:d-dot-gamma-suffices}
h(\sigma)^0 = h(d\dot\gamma)^0.
\end{equation}
Thus, from here on we can just take $\sigma = d \dot \gamma$ in~\eqref{eq:Cayley-h0-trace}, with $\dot \gamma$ as in~\eqref{eq:Cayley-test-variation}. (It is clear from its derivation that~\eqref{eq:Cayley-h0-trace} holds with $\sigma$ replaced by \emph{any} $4$-form.)

For this particular choice of $\dot \gamma$ in~\eqref{eq:Cayley-test-variation}, we need to verify that the additional condition $\int_M \rest{(\sta d \dot \gamma)}{M} = 0$ of equation~\eqref{eq:Cayley-condition} is indeed satisfied. We have
\begin{equation} \label{eq:Cayley-B-special}
\sta d \dot \gamma = - \sta d \sta (\sta \dot \gamma) = - e_i \hk \ol{\nabla}_{e_i} (\sta \dot \gamma) = - e_i \hk ( \sta \ol{\nabla}_{e_i} \dot \gamma).
\end{equation}
By~\eqref{eq:Cayley-test-variation} and Remark~\ref{rmk:lf}, we have
$$ \ol{\nabla}_{e_i} \dot \gamma = V \w W \w P(V, W, \ol{\nabla}_{e_i} \ol{\nabla} F) + \lf. $$
Thus by~\eqref{eq:Hessian-Fb}, along $M$ we have
$$ \ol{\nabla}_{e_i} \dot \gamma = V \w W \w P(V, W, e_i^{\perp}). $$
Since $V, W$ are \emph{tangent} along $M$, the above shows that $\ol{\nabla}_{e_i} \dot \gamma$ is of type $(3, 0) + (2, 1)$, where type $(p,q)$ denotes $p$ tangent directions and $q$ normal directions. Hence $\sta (\ol{\nabla}_{e_i} \dot \gamma)$ is of type $(1,4) + (2, 3)$. It follows that each term of $e_i \hk ( \sta \ol{\nabla}_{e_i} \dot \gamma)$ has at least one (in fact at least two) normal directions, thus it restricts to zero on $M$. By~\eqref{eq:Cayley-B-special}, we conclude that in fact $\rest{(\sta d \dot \gamma)}{M} = 0$, so certainly~\eqref{eq:Cayley-condition} holds.

The computations that led to equations~\eqref{eq:d-gamma-ijkl-h} and~\eqref{eq:d-gamma-jikl-h} in the proof of Theorem~\ref{thm:coassoc-B} used only the facts that $\chi$ and $\ps$ are skew. Since the same is true of $P$ and $\Ph$, the exact same computation yields
\begin{align}
& \qquad (d \dot \gamma) \big( P(e_j, e_k, e_l)^T, e_j, e_k, e_l \big) \nonumber \\
& = -12 \langle P(V, W, e_l)^{\perp}, P(V, W, e_l)^{\perp} \rangle - 6 \langle P(V, W, e_l^{\perp}), P(V, W, e_l)^T \rangle. \label{eq:d-gamma-ijkl-h-SP}
\end{align}
and
\begin{align}
& \qquad (d \dot \gamma) \big( P(e_j, e_k, e_l)^{\perp}, e_j, e_k, e_l \big) \nonumber \\
& = - 6 \langle P(V, W, e_l)^{\perp}, P(V, W, e_l)^{\perp} \rangle - 6 \langle P(V, W, e_l^{\perp}), P(V, W, e_l)^{\perp} \rangle. \label{eq:d-gamma-jikl-h-SP}
\end{align}

Now we substitute~\eqref{eq:d-gamma-ijkl-h-SP} and~\eqref{eq:d-gamma-jikl-h-SP} into the right hand side of~\eqref{eq:Cayley-h0-trace} for $\sigma = d \dot \gamma$, to obtain 
\begin{align*}
\tr_g h(d\dot\gamma)^0 & = \frac{1}{4} \langle P(V, W, e_l)^{\perp}, P(V, W, e_l)^{\perp} \rangle + \frac{1}{4} \langle P(V, W, e_l^{\perp}), P(V, W, e_l)^T \rangle \\
& \qquad {} - \frac{1}{4} \langle P(V, W, e_l^{\perp}), P(V, W, e_l)^{\perp} \rangle.
\end{align*}
Because $\{ e_1, e_2, e_3, e_4 \}$ are tangent and $\{ e_5, e_6, e_7, e_8 \}$ are normal, the above becomes
\begin{align*}
\tr_g h(d\dot\gamma)^0 & = \frac{1}{4} \sum_{a=1}^4 | P(V, W, e_a)^{\perp} |^2 + \frac{1}{4} \sum_{k=5}^8 | P(V, W, e_k)^{\perp} |^2 \\
& \qquad {} + \frac{1}{4} \sum_{l=5}^8 \langle P(V, W, e_l), P(V, W, e_l)^T \rangle - \frac{1}{4} \sum_{l=5}^8 | P(V, W, e_l)^{\perp} |^2 \\
& = \frac{1}{4} \sum_{a=1}^4 | P(V, W, e_a)^{\perp} |^2 + \frac{1}{4} \sum_{l=5}^8 | P(V, W, e_l)^T |^2.
\end{align*}
Noting that
$$ \sum_{l=5}^8 | P(V, W, e_l)^T |^2 = \sum_{l=5}^8 \sum_{a=1}^4 \langle P(V, W, e_l), e_a \rangle^2 = \sum_{a=1}^4 \sum_{l=5}^8 \langle P(V, W, e_a), e_l \rangle^2 = \sum_{a=1}^4 | P(V, W, e_a)^{\perp} |^2, $$
we deduce that
$$ \tr_g h(d\dot\gamma)^0 = \frac{1}{2} \sum_{a=1}^4 | P(V, W, e_a)^{\perp} |^2. $$
Hence from the above and~\eqref{eq:Cayley-B-hypothesis}, together with equations~\eqref{eq:h0-suffices} and~\eqref{eq:d-dot-gamma-suffices}, we finally obtain
\begin{equation*}
0 = \rest{\ddt}{t=0} \cV(\ol{g}_t) = \frac{1}{4} \int_M \left( \sum_{a=1}^4 | P(V, W, e_a)^{\perp} |^2 \right) \vol_{(M, g)}.
\end{equation*}
The integrand must vanish as it is nonnegative. Hence $P(V, W, e_a)^{\perp} = 0$ for $a = 1, 2, 3, 4$ and \emph{any tangent vector fields $V, W$ on $M$}. Taking $V, W$ to be $e_1, e_2, e_3, e_4$ shows that $P(e_a, e_b, e_c)$ is tangent to $M$ for $1 \leq a, b, c \leq 4$ and we conclude that $M$ is Cayley.
\end{proof}

\begin{rmk} \label{rmk:Cayley-differences} We saw in the proof of Theorem~\ref{thm:Cayley-A} why we needed to impose that $\sigma = \rest{\ddt}{t=0} \Ph_t = \pi_{35+7} d \dot \gamma$. Interestingly, in the almost complex, associative, and coassociative cases, not only do the  directions which are multiples of the calibration form not cause any problems for Theorem A, but they actually \emph{must be present} in order to obtain the proper cancellation in Theorems~\ref{thm:Um-A},~\ref{thm:assoc-A}, and~\ref{thm:coassoc-A}.

Similarly, by examining the proof of Theorem~\ref{thm:Cayley-B} we can see explicitly that the test variation~\eqref{eq:Cayley-test-variation}, which is precisely analogous to the test variations~\eqref{eq:Um-test-variation},~\eqref{eq:assoc-test-variation}, and~\eqref{eq:coassoc-test-variation} that worked in the almost complex, associative, and coassociative cases, respectively, \emph{does not work} in the Cayley case \emph{unless we project away} the component of $d \dot \gamma$ in the $\Omega^4_1$ direction. This is because if we do not drop the $\Omega^4_1$ term of $d \dot \gamma$, then we need to use $h$ instead of $h^0$, and thus we replace~\eqref{eq:Cayley-B-start} with~\eqref{eq:h-from-sigma}. The effect of replacing the factor of $\frac{1}{96}$ with $\frac{1}{112}$ destroys the careful cancellation. In fact, if one carefully checks the details, one finds that not dropping the $\Omega^4_1$ component of $d \dot \gamma$ results in
$$ \rest{\ddt}{t=0} \cV(\ol{g}_t) = \int_M \left(  \frac{1}{4} \sum_{a=1}^4 | P(V, W, e_a)^{\perp} |^2 + \frac{2}{7} | V \w W |^2 \right) \vol_{(M, g)} = 0. $$
This forces $V, W$ to be linearly dependent on $M$ and hence we cannot use $P(V, W, e_a)^{\perp} = 0$ to conclude anything in this case.

As we discussed at the start of this section, it seems that this problem (needing to drop the $\Omega^4_1$ component) may be somehow related to the fact that the space $\Omega^4_+$ of $\Spin{7}$-structures on $\ol{M}$ is \emph{not} an open subset of $\Omega^4$, unlike the case of $\G$-structures. This situation is reminiscent of a similar phenomenon encountered in~\cite[final two paragraphs of \S5.2]{K1}.
\end{rmk}

\begin{rmk} \label{rmk:Cayley-cannot-fix}
In this Cayley case, we had to impose a further condition on our allowed class of variations of the ambient metric, namely~\eqref{eq:Cayley-condition}, which is a condition that \emph{depends} on the submanifold $M$. A similar situation arose in the almost complex case when $k > 1$. (See parts (b) and (c) of Theorem~\ref{thm:Um-A}.) In that case, by imposing a condition on the torsion, specifically that $d \omega = 0$, we were able to remove this dependence on $M$. (See also Remark~\ref{rmk:Um-A-b}.)

It is therefore natural to consider whether imposing $d \Ph = 0$ (torsion-free $\Spin{7}$-structure) might ensure that~\eqref{eq:Cayley-condition} always holds when $M$ is Cayley. [This would mean that when $d\Ph = 0$, the requirement that the variations satisfy~\eqref{eq:Cayley-condition} can be dropped from Theorem~\ref{thm:Cayley-B} as well, since doing so results in a weaker statement.] However, it is easy to see that this is \emph{not} possible. The argument is as follows.

Consider a small ball $U$ in $M$ over which we have (oriented) local coordinates $x^1, \ldots, x^4$ and over which the normal bundle $NM$ trivializes with oriented orthonormal frame $e_5, \ldots, e_8$, inducing (oriented) local coordinates $x^1, \ldots, x^8$ on the total space of $NM$ over $U$, and hence, for $\sum_{k=5}^8 (x^k)^2 < \delta^2$, on an open set $V$ in $\ol{M}$ containing $U$ via the tubular neighbourhood theorem. By construction, on $U$ we have $x^k = 0$ and $\ddx{k} = e_k$ for $5 \leq k \leq 8$. It follows that, on $U$, we have $\sta (\dx{5} \w \dx{6} \w \dx{7} \w \dx{8}) = h \dx{1} \w \dx{2} \w \dx{3} \w \dx{4}$ where $h$ is some positive function.

Let $f \in C^{\infty}_c (U)$ be a bump function, so that $\int_M f > 0$. Let $\zeta \in C_c^{\infty} (B_{\delta}(0))$ be a cutoff function, with $\zeta = 1$ on $B_{\frac{1}{2} \delta} (0)$. Let $\dot \gamma = x^5 f(x^1, \ldots, x^4) \zeta(x^5, \ldots, x^8) \dx{6} \w \dx{7} \w \dx{8}$ on $V$, and extend it by zero to a smooth $3$-form on $\ol{M}$. By construction, on $U$ we have $\zeta = 1$ and also $x^k = 0$ and $\dx{k} = 0$ for $5 \leq k \leq 8$. It follows that, on $U$, we have $\sta (d \dot \gamma) = h f \dx{1} \w \dx{2} \w \dx{3} \w \dx{4}$ on $M$, so $\int_M \rest{(\sta d \dot \gamma)}{M} > 0$. [Note that this argument is very general and has nothing to do with Cayley submanifolds.]
\end{rmk}

\section{Further observations} \label{sec:further}

In this section we collect some further observations about our results.

\subsection{Alternate proof of Theorem A} \label{sec:thm-A-alt}

There is an alternate method of proof for Theorem A which is definitely simpler, but it does not allow one to determine the correct class of metric variations in the $\Spin{7}$ case as did our proof of Theorem~\ref{thm:Cayley-A}. Rather, this simpler proof works if one already knows the correct class of metric variations. (See also Remark~\ref{rmk:advantage-thm-A} below.) The alternate argument is as follows.

Let $(\ol{M}, \ol{g})$ be Riemannian, and let $\mu$ be a calibration form on $\ol{M}$ with respect to $\ol{g}$. Suppose that a path $\mu_t$ with $\mu_{t=0} = \mu$ induces a family of metrics $\ol{g}_t$ on $\ol{M}$ with $\ol{g}_{t=0} = \ol{g}$, such that $\mu_t$ is a calibration with respect to $\ol{g}_t$ for all $t \in (- \eps, \eps)$. Let $M$ be a $k$-dimensional compact embedded submanifold of $\ol{M}$ such that $M$ is calibrated by $\mu$, and let $g_t = \rest{\ol{g}_t}{M}$. Then we have
$$ \int_M \rest{\mu_t}{M} \leq \int_M \vol_{(M, g_t)}, \qquad \text{with equality at $t=0$,} $$
since the inequality holds on the level of the integrands. It follows that
$$ \rest{\ddt}{t=0} \cV(\ol{g}_t) = \rest{\ddt}{t=0} \int_M \rest{\mu_t}{M} = \int_M \rest{\dot \mu}{M}, $$
where $\dot \mu = \rest{\ddt}{t=0} \mu_t$. Thus if $\rest{\dot \mu}{M}$ is exact, then the right hand side is zero by Stokes's Theorem. This immediately proves Theorem A in the associative and coassociative cases, namely Theorems~\ref{thm:assoc-A} and~\ref{thm:coassoc-A}.

In the almost complex case, we have a family $(\ol{g}_t, J, \omega_t)$ of $\U{m}$-structures, so that $\mu_t = \frac{1}{k!} \omega_t^k$ are all calibrations with respect to $\ol{g}_t$. When $k=1$, we have
$$ \int_M \omega_t = \int_M \wt \omega_t = \int_M (\omega + d \alpha_t) = \int_M \omega = \text{ constant}, $$
where in the first equality we have used that the $(2,0) + (0,2)$ form $\rho_t$ restricts to zero on a $2$-dimensional almost complex submanifold. This is part (a) of Theorem~\ref{thm:Um-A}.

For $k > 1$, the calibration form is $\mu_t = \frac{1}{k!} \omega_t^k$ and hence
$$ \dot \mu_t = \frac{1}{(k-1)!} \omega_t^{k-1} \w \dot \omega_t. $$
Because the almost complex structure $J$ is fixed, $\dot \omega_t$ is of type $(1,1)$. Moreover, only a type $(k,k)$ form survives on restriction to a $2k$-dimensional almost complex submanifold. Using this fact and $\dot \omega_t = d \dot \alpha_t - \dot \rho_t$, we have
$$ \int_M \dot \mu_t = \frac{1}{(k-1)!} \int_M \omega_t^{k-1} \w \dot \omega_t = \frac{1}{(k-1)!} \int_M \omega_t^{k-1} \w d \dot \alpha_t. $$
If $\omega_t^{k-1} \w d \dot \alpha_t$ is exact, then the right hand side vanishes, yielding part (b) of Theorem~\ref{thm:Um-A}. This is automatic at $t=0$ if $d \omega = 0$, yielding part (c) of Theorem~\ref{thm:Um-A}.

In the $\Spin{7}$ case, we know from the proof of our Theorem~\ref{thm:Cayley-A} that the correct class of metric deformations correspond to paths $\Ph_t$ of $\Spin{7}$-structures such that $\rest{\ddt}{t=0} \Ph_t = \pi_{7+35} d \dot \gamma$, with the additional (submanifold dependent) condition~\eqref{eq:Cayley-condition} that $\int_M \rest{(\sta d \dot \gamma)}{M} = 0$. Since we know this now, the simpler method provides a quick verification, as follows. We have
$$ \rest{\ddt}{t=0} \cV( \ol{g}_t ) = \int_M \dot \Ph|_M = \int_M \rest{(\pi_{7+35} d \dot \gamma)}{M}. $$
Since $M$ is Cayley, one can show that any $\Omega^4_7$ form must restrict to zero on $M$. [See the sentence containing equation \eqref{eq:CayleyA-4}.] Then the right hand side above becomes
$$ \int_M \rest{(\pi_{35} d \dot \gamma)}{M} = \frac{1}{2} \int_M \rest{(d \dot \gamma - \sta d \dot \gamma)}{M} = 0 - \frac{1}{2} \int_M \rest{(\sta d \dot \gamma)}{M}, $$
which vanishes because of condition~\eqref{eq:Cayley-condition}.

\begin{rmk} \label{rmk:advantage-thm-A}
Another advantage of our case-by-case approaches to the proof of Theorem A is that it allows one to see explicitly the precise balancing of the $\pi_1$ and $\pi_{27}$ components of $\dot \mu$ in the associative and coassociative cases which makes the proof work, and which precisely \emph{does not} work in the Cayley case, as discussed in Remark~\ref{rmk:Cayley-differences}. It remains mysterious exactly why the Cayley case is so different. One possible avenue for exploring this is to analyze the properties of the infinitesimal ambient metric variations $h = \dot{\ol{g}}$ in terms of the decomposition of symmetric $2$-tensors on a Riemannian manifold (at least when $\ol{M}$ is compact) as described in Berger--Ebin~\cite{BE}.
\end{rmk}

\subsection{Relation to similar characterization of minimal submanifolds} \label{sec:related-minimal}

As before, fix a compact oriented embedded $k$-dimensional submanifold $M$ in the $n$-dimensional manifold $\ol{M}$, and consider again the functional
$$ \ol{g} \mapsto \cV(\ol{g}) = \int_M \vol_{(M,g)} \qquad \text{ where $g = \rest{\ol{g}}{M}$}, $$
on the space of Riemannian metrics $\ol{g}$ on $\ol{M}$.

Consider the class of ambient metric variations arising from diffeomorphisms isotopic to the identity. That is, consider variations of the form $\ol{g}_t = f_t^* \ol{g}$, where $f_t$ is the local one-parameter family of diffeomorphisms generated by a vector field $X$ on $\ol{M}$. By Lemma~\ref{lemma:first-variation}, $\ol{g}$ is a critical point of $\cV$ with respect to such variations if and only if $\frac{1}{2} \int_M (\tr_g h) \vol_{(M, g)} = 0$, where $h = \rest{\ddt}{t=0} \ol{g}_t$.

From $\ol{g}_t = f_t^* \ol{g}$, we have $\rest{\ddt}{t=0} \ol{g}_t = \mathcal L_X \ol{g}$, and thus
$$ h_{ij} = \ol{\nabla}_i X_j + \ol{\nabla}_j X_i. $$
Taking the trace in the tangential directions, we get
\begin{align*}
\frac{1}{2} \tr_g h & = \frac{1}{2} \sum_{a=1}^k h_{aa} = \sum_{a=1}^k \ol{\nabla}_a X_a = \sum_{a=1}^k \langle \ol{\nabla}_{e_a} X, e_a \rangle \\
& = \sum_{a=1}^k \langle \ol{\nabla}_a X^T + \ol{\nabla}_a X^{\perp}, e_a \rangle = \sum_{a=1}^k \langle \ol{\nabla}_{e_a} X^T, e_a \rangle - \sum_{a=1}^k \langle X^{\perp}, \ol{\nabla}_{e_a} e_a \rangle \\
& = \dive_g (X^T) - \langle X^{\perp}, \mathsf{H} \rangle,
\end{align*}
where $\dive_g$ is the Riemannian divergence on $(M, g)$ and $\mathsf{H}$ is the mean curvature vector field of $M$ in $(\ol{M}, \ol{g})$. Integrating both sides and using Stokes's Theorem, we deduce that $\ol{g}$ is a critical point of $\cV$ with respect to such variations if and only if $\mathsf{H} = 0$. (That is, if and only if $M$ is \emph{minimal}.)

Now suppose that $\ol{M}$ is equipped with a calibration form $\mu$, as in the four cases we considered in this paper. If $\mu_t = f_t^* \mu$, then in the associative, coassociative, and Cayley cases $\mu_t$ induces a variation of metrics $\ol{g}_t = f_t^* \ol{g}$. This is because $\mu$ determines the metric $\ol{g}$ in these three cases. In the almost complex case, it is easy to check using~\eqref{eq:Um-metric} that under the additional assumption that $df_t$ commutes with $J$ (which is fixed) for all $t$, then we get $\ol{g}_t = f_t^* \ol{g}$ in that case as well.

From $\dot \mu = \mathcal L_X \mu = d (X \hk \mu) + X \hk d \mu$, we see that these variations belong to the classes we considered in this paper if $d \mu = 0$. (Except possibly for the Cayley case.) This is consistent, because calibrated submanifolds are only guaranteed to be minimal if $d \mu = 0$. Thus, if the calibration form $\mu$ is closed, then the variations $\ol{g}_t$ induced from $\mu_t = f_t^* \mu$ are a proper subset of the variations we considered, and $\ol{g}$ is critical for this proper subset if and only if $M$ is minimal, whereas demanding that $\ol{g}$ be critical for the full class of variations we considered puts more constraints on $M$, and indeed we have shown that this happens if and only if $M$ is calibrated.

\subsection{Application to variational characterization of Smith maps} \label{sec:Smith}

In this section we apply our results to give a variational characterization of Smith maps. We first give a short introduction to Smith maps. (See~\cite{CKM} for more details.)

Let $\ol{M}$ be a smooth $n$-manifold, and let $M$ be a smooth compact oriented $k$-manifold. Given a smooth map $u \colon M \to \ol{M}$, its derivative $du$ is a smooth section of $T^* M \otimes u^* T \ol{M}$. In local coordinates $x^1, \ldots, x^k$ on $M$ and $y^1, \ldots, y^n$ on $\ol{M}$, we have
$$ du = \frac{\partial u^{\alpha}}{\partial x^i} dx^i \otimes u^* \frac{\partial}{\partial y^{\alpha}}. $$
Given Riemannian metrics $g$ and $\ol{g}$ on $M$ and $\ol{M}$, respectively, and a calibration $k$-form $\mu$ on $\ol{M}$ with respect to $\ol{g}$, we consider three different functionals involving $u$ and these geometric structures:
\begin{enumerate}[{$[$}1{$]$}]
\item From $g$ and $\ol{g}$ we get a metric $g^{-1} \otimes \ol{g}$ on $T^* M \otimes u^* T \ol{M}$, and thus we can compute the quantity
$$ |du|^2_{g, \ol{g}} = \langle du, du \rangle_{g^{-1} \otimes \ol{g}} = \frac{\partial u^{\alpha}}{\partial x^i} \frac{\partial u^{\beta}}{\partial x^j} g^{ij} (\ol{g}_{\alpha \beta} \circ u). $$
Since we have
$$ u^* \ol{g} = \frac{\partial u^{\alpha}}{\partial x^i} \frac{\partial u^{\beta}}{\partial x^j} (\ol{g}_{\alpha \beta} \circ u) dx^i dx^j, $$
we can also write
\begin{equation} \label{eq:norm-from-trace}
|du|^2_{g, \ol{g}} = \tr_g (u^* \ol{g}).
\end{equation}
That is, $u^* \ol{g}$ is a symmetric $2$-tensor on $M$, and $|du|^2_{g, \ol{g}}$ is its trace with respect to $g$. 

We define the \emph{$k$-energy} $\cE (u, g, \ol{g})$ of this triple of data to be
\begin{equation} \label{eq:k-energy}
\cE (u, g, \ol{g}) = \frac{1}{(\sqrt{k})^k} \int_M |du|^k_{g, \ol{g}} \, \vol_g.
\end{equation}
Note that there is in general no relation assumed between $g$ and $\ol{g}$. An important observation is that the $k$-energy depends only on the conformal class of $g$ on $M$. Explicitly, if we replace $g$ by $\lambda^2 g$, where $\lambda$ is a smooth positive function, then it is easy to see that the integrand $|du|^k_{g, \ol{g}} \, \vol_g$ is invariant.

\item Writing $\pa{u}{x^i}$ for $\pa{u^{\alpha}}{x^i} \paop{y^{\alpha}}$ and noting that the $k$-form $\big| \pa{u}{x^1} \w \cdots \w \pa{u}{x^k} \big|_{\ol{g}} \, dx^1 \w \cdots \w dx^k$ on $M$ is independent of the choice of local coordinates $x^1, \ldots, x^k$, we define the \emph{$k$-volume} $\cV(u, \ol{g})$ to be
\begin{equation} \label{eq:k-volume}
\cV(u, \ol{g}) = \int_M \Big| \pa{u}{x^1} \w \cdots \w \pa{u}{x^k} \Big|_{\ol{g}} \, dx^1 \w \cdots \w dx^k.
\end{equation}
It is easy to see that $u$ is an \emph{immersion} if and only if $u^* \ol{g}$ is a metric on $M$, in which case the integrand above is just $\vol_{u^* \ol{g}}$, and $\cV(u, \ol{g})$ is indeed the volume of the immersed submanifold $u(M)$ of $(\ol{M}, \ol{g})$. Below we denote the integrand of $\cV(u, \ol{g})$ by $\vol_{u^* \ol{g}}$ regardless of whether $u$ is an immersion.

\item In view of the previous point and recalling that our purpose here is to extend Theorems A and B to mappings, we also consider the following functional
$$ (u, \mu) \mapsto \int_{M}u^* \mu. $$
\end{enumerate}

These three functionals are related by two rather simple inequalities stated in the lemma below. The notion of a Smith map arises naturally from the conditions for equality.

\begin{lemma} \label{lemma:compare-cE-cV-mu}
The above three functionals are related as follows:
\begin{enumerate}[(a)]
\item $\cE (u, g, \ol{g}) \geq \cV (u, \ol{g})$. Equality holds if and only if $u^* \ol{g} = \frac{1}{k} |du|^2_{g, \ol{g}} \, g$ on $M$.
\item $\cV(u, \ol{g}) \geq \int_{M}u^* \mu$. Equality holds if and only if $u^* \mu = \vol_{u^* \ol{g}}$ on $M$.
\end{enumerate}
\end{lemma}
\begin{proof}
Take Riemannian normal coordinates at $p \in M$ with respect to $g$, so that $\{ \ddx{1}, \ldots, \ddx{k} \}$ are orthonormal at $p$. Define the $k \times k$ matrix $A$ by
$$ A_{ij} = (u^* \ol{g}) \Big( \ddx{i}, \ddx{j} \Big). $$
Then at $p$ we have
\begin{align*}
u^* \mu & = \mu\Big(\pa{u}{x^1}, \cdots, \pa{u}{x^k}\Big) \,\dx{1} \w \cdots \w \dx{k}, \\
\vol_{u^* \ol{g}} & = \sqrt{ \det A } \, \dx{1} \w \cdots \w \dx{k}, \\
\frac{1}{(\sqrt{k})^k}|du|^k_{g, \ol{g}} \, \vol_g & = \frac{1}{(\sqrt{k})^k} (\tr A)^{\frac{k}{2}} \, dx^1 \w \cdots dx^k.
\end{align*}
Since $\mu$ has comass one, we have
$$ \mu \Big( \pa{u}{x^1}, \cdots, \pa{u}{x^k} \Big) \leq \Big| \pa{u}{x^1} \w \cdots \w \pa{u}{x^k} \Big|_{\ol{g}} = \sqrt{\det A}. $$
This gives 
\begin{equation} \label{eq:v-mu-pointwise}
u^* \mu \leq \vol_{u^*\ol{g}},
\end{equation}
which integrates to the inequality in (b). The equality condition is obvious.

For (a), by the arithmetic-geometric mean inequality, we have $ (\det A)^{\frac{1}{k}} \leq \frac{1}{k} \tr A$, and thus
\begin{equation}\label{eq:am-gm}
\sqrt{\det A} \leq \frac{1}{(\sqrt{k})^k} (\tr A)^{\frac{k}{2}}.
\end{equation}
This gives
\begin{equation} \label{eq:e-v-pointwise}
\vol_{u^* \ol{g}} \leq \frac{1}{(\sqrt{k})^k} |du|^k_{g, \ol{g}} \, \vol_g, 
\end{equation}
which integrates to $\cV(u, \ol{g}) \leq \cE (u, g, \ol{g})$. Finally, $\cV(u, \ol{g}) = \cE (u, g, \ol{g})$ exactly when equality holds in~\eqref{eq:e-v-pointwise} everywhere on $M$, or equivalently when equality holds in~\eqref{eq:am-gm} for all $p \in M$. The latter happens if and only if all eigenvalues of the nonnegative definite matrix $A_{ij}$ coincide, in which case the latter is a nonnegative multiple of $\delta_{ij} = g_{ij}$ at $p$. In summary, $\cV(u, \ol{g}) = \cE (u, g, \ol{g})$ if and only if
$$ u^* \ol{g} = \lambda^2 g \text{ for some smooth function $\lambda^2$}. $$
Taking the trace of both sides with respect to $g$ and using~\eqref{eq:norm-from-trace}, we deduce that necessarily $\lambda^2 = \frac{1}{k} |du|^2_{g, \ol{g}}$. This completes the characterization of the equality case in (a).
\end{proof}

Regarding the equality condition in (a) above, we say that
\begin{equation} \label{eq:weakly-conformal}
u \colon (M, g) \to (\ol{M}, \ol{g}) \quad \begin{cases} \, \, \text{is \emph{weakly conformal} if} \, \, u^* \ol{g} = \frac{1}{k} |du|^2_{g, \ol{g}} \, g, \\
\, \, \text{is \emph{conformal} if it is weakly conformal and $|du|^2_{g, \ol{g}} > 0$ on $M$}. \end{cases}
\end{equation}
Note that a conformal map is necessarily an immersion.

As for the equality condition in part (b), noting that both $u^* \mu$ and $\vol_{u^* \ol{g}}$ vanish at the points where $u$ is not an immersion, we see that the condition is equivalent to the fact that the immersed submanifold $u(M')$ is calibrated by $\mu$, where $M' = \{ p \in M \mid du_p \text{ is injective} \}$.

We are now ready to define Smith maps. A map $u \colon (M, g) \to (\ol{M}, \ol{g}, \mu)$ is called a \emph{Smith map} if equality holds in both (a) and (b) of Lemma~\ref{lemma:compare-cE-cV-mu}, or equivalently if the following two conditions hold:
\begin{equation} \label{eq:Smith-map}
u^* \ol{g} = \frac{1}{k} |du|^2_{g, \ol{g}} \, g, \qquad \text{and} \qquad u^* \mu = \frac{1}{(\sqrt{k})^k} |du|^k_{g, \ol{g}} \, \vol_g.
\end{equation}

\begin{rmk} \label{rmk:Smith-maps}
The second condition in~\eqref{eq:Smith-map} is equivalent to $u^* \mu = \vol_{u^* \ol{g}}$ under weak conformality, since with $\lambda^2 = \frac{1}{k} |du|^2_{g, \ol{g}}$ we have
\begin{equation} \label{eq:conformal-map-volume}
\vol_{u^* \ol{g}} = \vol_{\lambda^2 g} = \lambda^k \vol_{g} = \frac{1}{(\sqrt{k})^k} |du|^k_{g, \ol{g}} \, \vol_g.
\end{equation}
We showed in the proof of Lemma~\ref{lemma:compare-cE-cV-mu} that the first condition in~\eqref{eq:Smith-map} is equivalent to $\vol_{u^* \ol{g}} = \frac{1}{(\sqrt{k})^k} |du|^{k}_{g, \ol{g}} \, \vol_g$, which in view of the inequalities~\eqref{eq:v-mu-pointwise} and~\eqref{eq:e-v-pointwise} is implied by the second condition in~\eqref{eq:Smith-map}. Thus the second condition in~\eqref{eq:Smith-map} alone defines Smith maps. A similar observation was made in~\cite[Example 1.3]{HPP}.
\end{rmk}

Next, suppose that $u \colon (M^k, g) \to (\ol{M}^n, \ol{g})$ is weakly conformal. Then it follows easily from the definition that $u$ is a Smith map if and only if $u^* \mu = \vol_{u^* \ol{g}}$. Hence if $du$ is never zero on $M$, then the image $u(M)$ of a Smith map is a submanifold of $\ol{M}$ that is \emph{calibrated} by $\mu$. (See~\cite[Section 3.1]{CKM}.)

In view of the results in Section~\ref{sec:calibrated}, it is natural to ask whether we can characterize Smith maps by studying $\cE(u, g, \ol{g})$ as a functional of $\ol{g}$. We are able to do this assuming in addition that $u$ is injective and conformal. See Theorems~\ref{thm:Smith-A} and~\ref{thm:Smith-B} below. For the sake of comparison, we mention that we can also consider $\cE (u, g, \ol{g})$ as a functional of either of its two other arguments. The following result is well-known, but we include a proof for completeness.
\begin{prop} \label{prop:energy-vary-domain}
Let $k \geq 2$. Consider $\cE (u, g, \ol{g})$ as a functional of the domain metric $g$. Then $g$ is a critical point for $\cE$ if and only if $u$ is weakly conformal.
\end{prop}
\begin{proof}
Let $g_t$ be a family of metrics on $M$ with $g_{t=0} = g$, and let $h = \rest{\ddt}{t=0} g_t$. Then we have $\rest{\ddt}{t=0} g_t^{ij} = - h^{ij}$ and $\rest{\ddt}{t=0} \vol_{g_t} = \frac{1}{2} (\tr_g h) \vol_g$. Using these facts and equation~\eqref{eq:norm-from-trace}, we compute
\begin{align*}
\rest{\ddt}{t=0} \cE (u, g_t, \ol{g}) & = \frac{1}{(\sqrt{k})^k} \int_M \rest{\ddt}{t=0} \Big( (|du|^2_{g_t, \ol{g}})^{\frac{k}{2}} \vol_{g_t} \Big) \\
& = \frac{1}{(\sqrt{k})^k} \int_M \Big( \frac{k}{2} (|du|^2_{g, \ol{g}})^{\frac{k}{2} - 1} (- \langle h, u^* \ol{g} \rangle_g) \vol_g + (|du|^2_{g, \ol{g}})^{\frac{k}{2}} \frac{1}{2} (\tr_g h) \vol_g \Big) \\
& = \frac{1}{2 (\sqrt{k})^k} \int_M \langle h, - k |du|^{k-2}_{g, \ol{g}} u^* \ol{g} + |du|^k_{g, \ol{g}} g \rangle_g \vol_g.
\end{align*}
The above vanishes for all $h$ if and only if $|du|^{k-2}_{g, \ol{g}} ( - k u^* \ol{g} + |du|^2_{g, \ol{g}} g) = 0$, which is clearly equivalent to the weakly conformal condition of~\eqref{eq:weakly-conformal}.
\end{proof}

Next consider $\cE (u, g, \ol{g})$ as a functional of the map $u \colon M \to \ol{M}$. It is well-known that the critical points in this case are the \emph{$k$-harmonic maps}, which are the maps that satisfy the $k$-harmonic map equation
$$ \dive_g ( |du|^{k-2}_{g, \ol{g}} du ) = 0. $$
If $d \mu = 0$, any Smith map $u \colon (M^k, g) \to (\ol{M}^n, \ol{g}, \mu)$ is a $k$-harmonic map. (See~\cite[Section 3.5]{CKM}.) By this fact and Proposition~\ref{prop:energy-vary-domain}, we see that if $(u, g, \ol{g})$ is a Smith map, then $(u, g)$ is critical for $\cE(\cdot, \cdot, \ol{g})$. However, there are weakly conformal $k$-harmonic maps that are not Smith maps. The next two results, for which we do not assume that $d \mu = 0$, are analogues of Theorems A and B in the context of Smith maps. They essentially say (modulo the assumption that $u$ is injective and conformal) that a Smith map $(u, g, \ol{g})$ is \emph{characterized} by the condition that $\ol{g}$ is critical for $\cE(u, g, \cdot)$, with respect to the special class of ambient metric variations $\ol{g}_t$ considered in Section~\ref{sec:calibrated}. 

As preparation, we note that if $\ol{g}_t$ is a family of metrics on $\ol{M}$ with $\ol{g}_{t=0} = \ol{g}$ and $\ol{h} = \rest{\ddt}{t=0} \ol{g}_t$, then computing as we did in Proposition~\ref{prop:energy-vary-domain} yields that
\begin{equation} \label{eq:energy-vary-target}
\rest{\ddt}{t=0} \cE (u, g, \ol{g}_t) = \frac{k}{2 (\sqrt{k})^k} \int_M |du|^{k-2}_{g, \ol{g}} \langle u^* \ol{h}, g \rangle_g \vol_g.
\end{equation}

\begin{thm}[Theorem A for Smith maps] \label{thm:Smith-A}
Let $u \colon (M, g) \to (\ol{M}, \ol{g}, \mu)$ be an injective Smith map. Suppose that $|du|^2_{g, \ol{g}} > 0$ on $M$, so that $u$ is conformal. Then $\ol{g}$ is a critical point of the $k$-energy function $\cE (u, g, \ol{g})$, thought of as a functional of $\ol{g}$, with respect to the special class of ambient metric variations $\ol{g}_t$ considered in Section~\ref{sec:calibrated}. 
\end{thm}
\begin{proof}
Let $\ol{h} = \rest{\ddt}{t=0} \ol{g}_t$. Since $u$ is assumed conformal, by~\eqref{eq:conformal-map-volume}, and that fact that $|du|^2_{g, \ol{g}} > 0$, we have
$$ \frac{1}{(\sqrt{k})^k} |du|^{k-2}_{g, \ol{g}} \, \vol_g = \frac{1}{|du|^2_{g, \ol{g}}} \, \vol_{u^* \ol{g}}. $$
Moreover, from $u^* \ol{g} = \frac{1}{k} |du|^2_{g, \ol{g}} \, g$, we get $g^{-1} = \frac{1}{k} |du|^2_{g, \ol{g}} \, (u^* \ol{g})^{-1}$, and thus
$$ \langle u^* \ol{h}, g \rangle_g = \tr_g (u^* \ol{h}) = \frac{1}{k} |du|^2_{g, \ol{g}} \tr_{u^* \ol{g}} (u^* \ol{h}). $$
Substituting the above two equations into~\eqref{eq:energy-vary-target}, we obtain
$$ \rest{\ddt}{t=0} \cE (u, g, \ol{g}_t) = \frac{1}{2} \int_M \tr_{u^* \ol{g}} (u^* \ol{h}) \, \vol_{u^* \ol{g}}. $$
By Lemma~\ref{lemma:first-variation}, the above says
$$ \rest{\ddt}{t=0} \cE (u, g, \ol{g}_t) = \rest{\ddt}{t=0} \cV (u, \ol{g}_t). $$
Now since $u$ is injective and conformal, and since $M$ is compact, we see that $u$ is an embedding. Moreover, since $u$ is a Smith map, the image $u(M)$ is calibrated by $\mu$. The result now follows by Theorem A for calibrated submanifolds.
\end{proof}

\begin{thm}[Theorem B for Smith maps] \label{thm:Smith-B}
Let $u \colon (M, g) \to (\ol{M}, \ol{g}, \mu)$ be an injective conformal map. Suppose that $\ol{g}$ is a critical point of the $k$-energy function $\cE (u, g, \ol{g})$, thought of as a functional of $\ol{g}$, with respect to the special class of ambient metric variations $\ol{g}_t$ considered in Section~\ref{sec:calibrated}. Then $u$ is a Smith map.
\end{thm}
\begin{proof}
Again we have that $u$ is an embedding. Moreover, as in the previous proof we have by the conformality of $u$ that
$$ \rest{\ddt}{t=0} \cE (u, g, \ol{g}_t) = \rest{\ddt}{t=0} \cV (u, \ol{g}_t). $$
The left hand side vanishes by assumption, and the result now follows by Theorem B for calibrated submanifolds.
\end{proof}

\subsection{Future questions} \label{sec:future}

There are several directions in which one could continue these investigations:

\begin{itemize}
\item One could try to relax the assumption that $M$ is \emph{embedded} in $\ol{M}$ to an immersion $\iota \colon M \to \ol{M}$. We needed $\iota$ to be an embedding to ensure that the function $F$ in Definition~\ref{defn:F} was well-defined. If this succeeds one could then further try to generalize Theorems~\ref{thm:Smith-A} and~\ref{thm:Smith-B} to apply to maps that are only weakly conformal, as discussed in Section~\ref{sec:Smith}.
\item An obvious question is to consider the case of \emph{special Lagrangian submanifolds} in manifolds $\ol{M}^{2m}$ with $\SU{m}$-structure. This case seems to be more complicated, for the following reasons. First, an $\SU{m}$-structure is determined by a pair of differential forms: a real $(1,1)$-form $\omega$ and a complex $(m,0)$ form $\Omega$, related by $\omega^m = c_m \Omega \w \ol{\Omega}$, where $c_m$ is some universal constant depending on the ``complex'' dimension $m$. The real $m$-forms $\mu_{\theta} = \real (e^{i \theta} \Omega)$ are calibrations for all $\theta \in [0, 2 \pi)$. Should the right thing to do be to vary $\mu_{\theta}$ in the direction of exact forms, or to vary $\Omega$ in the direction of exact forms? Either way, we are forced to also vary $\omega$ in order to maintain an $\SU{m}$-structure at all times. It is not clear exactly how to do this in a natural way.
\item One might try to look for a general argument for Theorem B that works for any calibration form. However, the fact that the Cayley case is so different from the others hints that any such general argument may be somewhat subtle. In particular it appears that, even if one assumes that the calibration form $\mu$ is closed, varying $\mu$ to stay in a fixed cohomology class does not seem to be natural, at least in the Cayley case.
\item It is natural to consider the \emph{second variation} of the functional $\ol{g} \mapsto \cV(\ol{g})$ with respect to variations in our subclass $\cG$ that makes Theorems A and B true. In particular, are the critical metrics $\ol{g}$ \emph{stable} critical points? Note that in~\cite{AZ1,AZ2}, Arezzo--Sun did compute the second variation in the almost K\"ahler case, and showed that if the functional $\cV$ had a \emph{stable point} $\ol{g}$ (that is, if $\frac{d^2}{dt^2} \big|_{t=0} \cV (\ol{g}_t) \geq 0$, even if $\ol{g}$ is not a critical point), then the embedded submanifold is calibrated. This is interesting, but does not answer the question of whether all critical points are stable.
\item Finally, one could consider some kind of \emph{gradient flow} of $\cV$, to try to flow the pair $(\ol{g}, \mu)$ on $\ol{M}$ to a limiting pair $(\ol{g}_{\infty}, \mu_{\infty})$ where $\mu_{\infty}$ is a calibration with respect to $\ol{g}_{\infty}$ and $M$ is calibrated by $\mu_{\infty}$. However, the objects flowing would be on $\ol{M}$, and the gradient flow would depend only on their restrictions to the submanifold $M$, so it is not clear how to deal with this problem.
\end{itemize}

\end{document}